\documentclass[onefignum,onetabnum]{siamart251216}

\usepackage{amsfonts,amssymb,mathtools}
\usepackage{booktabs}
\usepackage{enumitem}
\usepackage{graphicx}
\usepackage{microtype}
\usepackage{xcolor}
\usepackage{url}

\headers{Feedback Cycles in Exploratory Equilibria}{Chen-Hung Wu}

\title{Feedback Cycles in Exploratory Equilibria\thanks{\funding{This work received
no external funding.}}}
\author{Chen-Hung Wu\thanks{Independent Researcher, Redwood City, CA
  (\email{chenhung.wu1572@gmail.com},
  \url{https://orcid.org/0009-0008-1026-1083}).}}

\hypersetup{
  pdftitle={Feedback Cycles in Exploratory Equilibria},
  pdfauthor={Chen-Hung Wu},
  pdfsubject={Exploratory equilibria in time-inconsistent stochastic control},
  pdfkeywords={stochastic control, reinforcement learning, equilibrium HJB,
    entropy regularization, Volterra operator, feedback cycles}
}

\newsiamthm{assumption}{Assumption}
\newsiamremark{remark}{Remark}
\newsiamremark{example}{Example}
\AddToHook{env/assumption/begin}{\crefalias{theorem}{assumption}}
\crefname{assumption}{Assumption}{Assumptions}
\Crefname{assumption}{Assumption}{Assumptions}

\newcommand{\A}{\mathsf A}
\newcommand{\Pp}{\mathbb P}
\newcommand{\E}{\mathbb E}
\newcommand{\R}{\mathbb R}
\newcommand{\KL}{\operatorname{KL}}
\newcommand{\osc}{\operatorname{osc}}
\newcommand{\cG}{\mathcal G}
\newcommand{\cQ}{\mathcal Q}
\newcommand{\cK}{\mathcal K}
\newcommand{\cS}{\mathcal S}
\newcommand{\cX}{\mathcal X}
\newcommand{\cY}{\mathcal Y}
\newcommand{\norm}[1]{\left\lVert #1\right\rVert}

\begin{document}
\maketitle

\begin{abstract}
Entropy regularization smooths equilibrium policies in time-inconsistent stochastic
control.  At low temperature, the same Gibbs response can strongly amplify errors in
learned rewards and dynamics.  We show that the derivative of an exploratory equilibrium
is governed by a backward Volterra--parabolic resolvent.  Along an aligned positive
mode, a lower bound has the same exponential order.  A block
decomposition identifies the source of the amplification: causal paths contribute powers
of $1/\tau$, whereas a positive feedback cycle can produce exponential growth.

At fixed temperature, a local equilibrium branch is twice differentiable with respect to
finite-dimensional model parameters, which yields a function-valued delta method.  A
bounded uniformly elliptic diffusion realizes this path--cycle distinction in every
finite dimension.  Closing one positive cycle changes the root-$n$ linear-response boundary from
a power law to order $1/\log n$; along the cyclic Perron mode, right-endpoint
discretization is relatively consistent exactly when $N\tau^2\to\infty$.  An affine model
also gives an exact nonlinear transition at the Lambert-$W$ temperature
$\beta T/W(\beta T\sqrt n)$.  Numerical calculations illustrate these rates.
\end{abstract}



\section{Introduction}\label{sec:introduction}

Entropy regularization replaces a point action by a relaxed policy and turns policy
improvement into a Gibbs map.  In continuous-time reinforcement learning (RL), this
gives a smooth Hamiltonian and a tractable policy iteration.  Convergence results for
time-consistent exploratory HJB equations include
\cite{tang2022exploratory_hjb_sicon,huang2025convergence,tran2025policy,
ma2025convergence}.  Entropy annealing and robustness have also been studied for
time-consistent continuous-time control
\cite{SethiSiskaZhang2025,CaoAceroSiskaZhang2026}.

For a time-inconsistent objective, a policy chosen at $(t,x)$ need not remain optimal
at a later time.  One instead seeks an equilibrium among temporal selves, described by
an extended HJB system
\cite{Yong2012,Bjork2017,BjorkKhapkoMurgoci2021,lei_nonlocal_2023,
lei_nonlocality_2024}.  That system is nonlocal: the policy at time $t$ depends on a
diagonal derivative of an auxiliary value function whose reference variables are frozen
at $(t,x)$.

Huang, Yu, and Zhang~\cite{HuangYuZhang2026} recently proposed an exploratory
equilibrium HJB (EEHJB) system for a broad class of these problems, proved
fixed-temperature convergence of policy iteration, and constructed a classical
solution.  A related vanishing-entropy result obtains unregularized equilibria from
subsequential EEHJB limits \cite{WangYuZhangZhou2026}.  These works address
well-posedness, fixed-model policy iteration, and the zero-temperature limit, but not
how learned-model or policy-evaluation errors are amplified as the temperature falls.
A fixed-temperature iteration can converge quickly while its equilibrium is poorly
conditioned.

Our analysis is self-contained relative to that well-posedness theory.  We state the
fixed-point system and its data hypotheses in
\cref{sec:model,ass:eehjb-data,ass:eehjb-envelope}, construct the mild evaluation map
and local equilibrium branch directly, and solve the explicit models separately.

Adjacent perturbation theory covers time-consistent relaxed control
\cite{reisinger2021regularity} and equilibrium-induced values for time-inconsistent
stopping \cite{BayraktarWangZhou2023,BayraktarWangZhou2022}.  Nonlocal equilibrium HJB
equations, equilibrium dynamic programming, and related BSDE representations are
studied in \cite{lei2023wellposedness,PossamaiRodriguezPolo2026}.  Relaxed equilibria
and sample-based learning for time-inconsistent control and decision problems appear in
\cite{BayraktarHuangWangZhou2025,dai2023learning,LesmanaPun2025,GuoHuangYu2026}.
None of these results gives a temperature-explicit derivative for a continuous-time
exploratory control equilibrium and propagates it through a model estimator.

The organizing point is that low-temperature conditioning is governed not only by the
local $1/\tau$ Gibbs factor, but also by the topology and smoothing of causal returns.
A future policy changes the continuation score of every earlier self.  In discrete time
this derivative is strictly triangular.  In continuous time it becomes a backward
Volterra operator with infinitely many nonzero causal powers.  The series converges for
each $\tau>0$, but its norm may grow exponentially as $\tau\downarrow0$.

Treating future decision nodes as agents with logit responses goes back to agent
quantal response equilibrium in extensive-form games \cite{McKelveyPalfrey1998}.
Static graph QRE can also exhibit Ising-type phase changes
\cite{LeonidovSavvateevSemenov2019}.  The analogy here is limited: temporal selves form
a continuum coupled by a controlled diffusion, and the directed graph describes a
causal derivative rather than equilibrium multiplicity.

Let $\alpha$ denote the smoothing order of the causal kernel.  The equilibrium derivative
is bounded by $C\tau^{-1}E_\alpha(\kappa T^\alpha/\tau)$.  If the forcing and return
kernel preserve an aligned sign, the lower bound has the same small-temperature
exponent.  Without that sign condition, negative feedback may damp the response.

The block form shows where repeated amplification enters.  On a directed acyclic graph,
the Neumann series stops at the longest path.  Under aligned forcing, a positive
$q$-cycle with total smoothing order $\beta$ contributes the lower factor
$E_\beta(C/\tau^q)$ and forces divergence below the scale
$(\log n)^{-\beta/q}$.  When the corresponding one-step upper bound matches the cycle
return, this scale is the linear-response boundary.

To connect conditioning with estimation, we treat a finite-dimensional coefficient
family at fixed $\tau>0$.  The mild EEHJB evaluation map is $C^2$, its causal derivative
is invertible, and the resulting equilibrium branch satisfies a function-valued delta
method.  The parabolic argument gives the general upper scale $e^{C/\tau^2}$; whether a
state-dependent model attains this rate remains open.

Two exact models show what the causal rates mean.  In a bounded uniformly elliptic
diffusion of arbitrary finite dimension, the tied-policy response is
\[
 \frac v\tau\exp\left\{\left(-\nu I+\frac v\tau K\right)(T-t)\right\}.
\]
For $K\ge0$, a graph of longest path $L$ has susceptibility
$\Theta(\tau^{-(L+1)})$, while a positive cycle produces an $e^{C/\tau}$ factor
along a Perron direction.  The model also has a nonlinear counterpart: exponentially
small signals vanish on a DAG, whereas cooperative feedback selects extreme actions
under the positive-row-sum, support, and short-horizon conditions of
\cref{cor:nonlinear-graph}.  In the affine model,
root-$n$ reward noise has the exact critical temperature
\[
 \tau_n^\star=\frac{\beta T}{W(\beta T\sqrt n)}
       \sim\frac{2\beta T}{\log n}.
\]
Here $W$ denotes the principal real branch of the Lambert $W$ function.
At this scale the initial policy has a nondegenerate random limit although every fixed
later policy converges to the reference mixture.

Time discretization retains this graph dependence.  Along a cyclic Perron mode, a
uniform $N$-step rule has vanishing relative error exactly when
$N\tau^2\to\infty$.  For a fixed DAG, $N\to\infty$ suffices without coupling $N$ to
$\tau$.  The general stability and delta-method results keep $\tau>0$ fixed; joint
small-temperature limits are stated only when an explicit bound or exact model supports
them.

The paper follows this causal chain.  \Cref{sec:model} derives the Gibbs fixed point
from infinitesimal spikes.  \Cref{sec:volterra} proves the abstract resolvent and graph
dichotomy.  \Cref{sec:stability} establishes EEHJB stability and the fixed-temperature
delta method.  \Cref{sec:exact} realizes the sharp path--cycle and nonlinear statistical
transitions, and \cref{sec:discretization} closes with discretization and reproducible
numerical illustrations.

\section{Time-inconsistent relaxed control and the EEHJB score map}
\label{sec:model}

Fix a horizon $T>0$, state dimension $d$, and a compact metric action space $\A$ equipped with
a reference probability measure $\mu$ of full support.  A Markov relaxed policy is a
kernel $\pi(t,x,da)$ absolutely continuous with respect to $\mu$.  Put
\[
 \bar b^\pi(t,x)=\int_\A b(t,x,a)\pi(t,x,da),
 \qquad
 \bar r_{\tau,y}^\pi(t,x)=\int_\A r(\tau,y,t,x,a)\pi(t,x,da).
\]
Finite action spaces are included: if $\mu$ gives positive mass to each action, every
policy has a density with respect to $\mu$, and the integrals and relative entropy
below are finite sums.
For a known, uncontrolled diffusion matrix $\sigma(t,x)$, the exploratory state is
\begin{equation}\label{eq:state-sde}
 dX_s^\pi=\bar b^\pi(s,X_s^\pi)\,ds+\sigma(s,X_s^\pi)\,dW_s,
 \qquad X_t^\pi=x.
\end{equation}
The self whose reference time-state is $(\tau,y)$ evaluates a continuation from $(t,x)$
by
\begin{align}
 V^\pi(\tau,t,y,x)
  :=\E_{t,x}\bigg[&\int_t^T
   \Big\{r(\tau,y,s,X_s^\pi,A_s)
       -\tau_e\log\frac{d\pi(s,X_s^\pi)}{d\mu}(A_s)\Big\}\,ds\notag\\
       &+F(\tau,y,X_T^\pi)\bigg].                     \label{eq:aux-value}
\end{align}
Here, conditionally on the state, $A_s$ has distribution $\pi(s,X_s^\pi,\cdot)$.
We write the entropy temperature as $\tau_e>0$ temporarily to distinguish it from the
reference-time variable; from \cref{sec:volterra} onward it is denoted simply by
$\tau$.  Nonexponential discounting can be absorbed into $r$.

The policy-evaluation equation for a fixed admissible $\pi$ is the family of linear
parabolic PDEs
\begin{align}
 \partial_tV^\pi+\tfrac12\operatorname{tr}
  (\sigma\sigma^\top D^2_xV^\pi)
 +\bar b^\pi\cdot D_xV^\pi+\bar r_{\tau,y}^\pi
 -\tau_e\KL(\pi(t,x)\Vert\mu)&=0,                     \label{eq:policy-evaluation-pde}\\
 V^\pi(\tau,T,y,x)&=F(\tau,y,x).                      \notag
\end{align}
The variables $(\tau,y)$ are frozen parameters in this PDE.  The acting self at $(t,x)$
sets $(\tau,y)=(t,x)$ but uses the \emph{flow-state partial derivative}
$D_xV^\pi(t,t,x,x)$, not the total derivative of the diagonal map.  Its action score is
\begin{equation}\label{eq:score-map}
 Q^\pi(t,x,a)=b(t,x,a)\cdot D_xV^\pi(t,t,x,x)
                 +r(t,x,t,x,a).
\end{equation}
The Gibbs variational identity gives the unique pointwise best response
\begin{equation}\label{eq:gibbs-map}
 \cG_{\tau_e}(q)(da)
 =\frac{\exp(q(a)/\tau_e)\mu(da)}
        {\int_\A\exp(q(a')/\tau_e)\mu(da')}.
\end{equation}
Define the best-response operator
\begin{equation}\label{eq:equilibrium-fixed-point}
 \mathcal T_{\tau_e,M}(\pi):=\cG_{\tau_e}(\cQ_M(\pi)),
 \qquad M=(b,r,F).
\end{equation}
A regularized Markov equilibrium is a fixed point
$\pi^\star=\mathcal T_{\tau_e,M}(\pi^\star)$; when it is unique we write
$\Phi_{\tau_e}(M)=\pi^\star$.  Here $\cQ_M$ means: solve
\eqref{eq:policy-evaluation-pde} and form \eqref{eq:score-map}.  This is a compact
score-map representation of the EEHJB system.
The stability section adds a terminal statistic and a nonlinear function $G$ and
states that extended system explicitly.

The next result anchors this fixed-point formulation in the infinitesimal-spike
equilibrium criterion.  Its local hypotheses justify the first-order spike limit; the
conclusion identifies the Gibbs fixed point as exactly the no-profitable-spike condition.

\begin{proposition}[Infinitesimal spike characterization]
\label{prop:spike-verification}
Fix $(t,x)\in[0,T)\times\R^d$.  Let $\rho\ll\mu$ satisfy
$\KL(\rho\Vert\mu)<\infty$, and set
\[
 \pi^{\epsilon,\rho}(s,z)=
 \begin{cases}
  \rho, &t\le s<t+\epsilon,\\
  \pi(s,z), &t+\epsilon\le s\le T.
 \end{cases}
\]
Let $J_{t,x}(\eta):=V^\eta(t,t,x,x)$ denote the payoff of the self whose
reference variables are frozen at $(t,x)$.  Suppose that, for some $\delta>0$,
$u(s,z):=V^\pi(t,s,x,z)$ belongs to
$C^{1,2}([t,t+\delta]\times\R^d)$ and satisfies
\eqref{eq:policy-evaluation-pde} on this strip.  Assume that $b$, $\sigma$, and $r$ are
continuous at $(t,x)$, uniformly in $a\in\A$, that the spike SDE is locally
well posed, and that $u$ and its derivatives in It\^o's formula have
polynomial growth controlled by corresponding moments of the spike process.
Assume also that
\[
 (s,z)\longmapsto
 \big(\bar b^\pi(s,z),\bar r_{t,x}^\pi(s,z),
       \KL(\pi(s,z)\Vert\mu)\big)
\]
is right-continuous at $(t,x)$.  Assume $\KL(\pi(t,x)\Vert\mu)<\infty$,
$Q^\pi(t,x,\cdot)\in C(\A)$, and, for some $\epsilon_0\in(0,\delta]$, the entire
bracketed integrand in \eqref{eq:exact-spike-difference} is uniformly integrable over
$0<\epsilon<\epsilon_0$ and $s\in[t,t+\epsilon]$.  Then
\begin{equation}\label{eq:spike-gain}
 \begin{aligned}
  \lim_{\epsilon\downarrow0}
  \frac{J_{t,x}(\pi^{\epsilon,\rho})-J_{t,x}(\pi)}{\epsilon}
  &=\Gamma_{t,x}^\pi(\rho)-\Gamma_{t,x}^\pi(\pi(t,x)),\\
  \Gamma_{t,x}^\pi(\rho)
  &=\int_\A Q^\pi(t,x,a)\rho(da)-\tau_e\KL(\rho\Vert\mu).
 \end{aligned}
\end{equation}
If these hypotheses hold for every finite-entropy $\rho$ at each point, then $\pi$
admits no infinitesimal spike with positive limit in \eqref{eq:spike-gain} if and only
if it satisfies \eqref{eq:equilibrium-fixed-point}.  The maximizer is unique up to
$\mu$-null sets.
\end{proposition}

\begin{proof}
Let $X^{\epsilon,\rho}$ be the state process under the spike policy, and write
\begin{align*}
 \bar b^\rho(s,z)&=\int_\A b(s,z,a)\rho(da),&
 \bar r_{t,x}^\rho(s,z)&=\int_\A r(t,x,s,z,a)\rho(da),\\
 k^\pi(s,z)&=\KL(\pi(s,z)\Vert\mu).&&
\end{align*}
Since the policy returns to $\pi$ at $t+\epsilon$, the Markov property gives
\begin{align*}
 J_{t,x}(\pi^{\epsilon,\rho})
 =\E\bigg[&\int_t^{t+\epsilon}
 \big\{\bar r_{t,x}^\rho(s,X_s^{\epsilon,\rho})
       -\tau_e\KL(\rho\Vert\mu)\big\}\,ds\\
 &+u(t+\epsilon,X_{t+\epsilon}^{\epsilon,\rho})\bigg].
\end{align*}
For $0<\epsilon<\delta$, apply It\^o's formula to
$u(s,X_s^{\epsilon,\rho})$ and use the
policy-evaluation equation for $u$.  Because the diffusion matrix in
\eqref{eq:state-sde} is independent of the action, the second-order terms in
the two generators cancel exactly.  We obtain
\begin{align}
 J_{t,x}(\pi^{\epsilon,\rho})-u(t,x)
 =\E\int_t^{t+\epsilon}\bigg[&
 (\bar b^\rho-\bar b^\pi)\cdot D_zu
 +(\bar r_{t,x}^\rho-\bar r_{t,x}^\pi)\notag\\
 &-\tau_e\big\{\KL(\rho\Vert\mu)
              -k^\pi\big\}
 \bigg](s,X_s^{\epsilon,\rho})\,ds .
 \label{eq:exact-spike-difference}
\end{align}
Here $k^\pi$ is evaluated at $(s,X_s^{\epsilon,\rho})$, whereas
$\KL(\rho\Vert\mu)$ is constant on the spike interval.  A stopped
Burkholder--Davis--Gundy estimate, recorded in the supplement, gives
$\sup_{t\le s\le t+\epsilon}|X_s^{\epsilon,\rho}-x|\to0$ in probability.
Right-continuity identifies the limit of the bracketed integrand, and its assumed
uniform integrability permits passage to the limit by Vitali's theorem in
\eqref{eq:exact-spike-difference}.  Dividing by $\epsilon$ gives
\eqref{eq:spike-gain}, now with the baseline drift, reward, and entropy terms
displayed explicitly.

Set $Z=\int_\A e^{Q^\pi(t,x,a)/\tau_e}\mu(da)$.  Compactness of $\A$ and
continuity of $Q^\pi(t,x,\cdot)$ give $0<Z<\infty$.  If
$G=\cG_{\tau_e}(Q^\pi(t,x,\cdot))$, direct calculation yields
\[
 \Gamma_{t,x}^\pi(\rho)
 =\tau_e\log Z-\tau_e\KL(\rho\Vert G).
\]
Thus $G$ is the unique maximizer.  Requiring the derivative in
\eqref{eq:spike-gain} to be nonpositive for every finite-entropy $\rho$ is
equivalent to $\pi(t,x)=G$.  Applying this pointwise proves the fixed-point
equivalence.
\end{proof}

This equivalence concerns the local equilibrium criterion defined by
infinitesimal spikes.  It does not assert precommitment optimality or the
strong finite-spike equilibrium property.

A perturbation of $\pi(s,\cdot)$ at $s>t$ changes
$V^\pi(t,t,\cdot,\cdot)$, hence the score and policy at $t$.  Causality prevents a
policy perturbation at an earlier time from changing a later score.  Thus
$D_\pi\cQ_M$ is a backward Volterra operator.

\section{The causal Volterra resolvent}
\label{sec:volterra}

We first separate the causal argument from the PDE estimates.  Let
$X$ and $Y$ be Banach spaces for policy tangents and centered action scores at one
time, respectively, and put $\cX=C([0,T];X)$ and $\cY=C([0,T];Y)$ with supremum
norms.  In finite actions, $X$ is the product of simplex tangent spaces with the
$\ell^1$ norm and $Y$ carries the oscillation norm.  For a signed measure $\nu$, our total
variation convention is $\norm{\nu}_{\rm TV}=|\nu|(\A)$, so the distance between two
probability measures is at most two.
Let $\mathsf H$ be a Banach space of primitive model directions.

For $\alpha>0$, define the right-sided Riemann--Liouville integral
\begin{equation}\label{eq:fractional-integral}
 (I_{T-}^\alpha f)(t)
 :=\frac1{\Gamma(\alpha)}\int_t^T(s-t)^{\alpha-1}f(s)\,ds.
\end{equation}
We use the convention $I_{T-}^0=\mathrm{Id}$.  The fractional integrals satisfy
$I_{T-}^\alpha I_{T-}^\gamma=I_{T-}^{\alpha+\gamma}$ and, for every integer
$m\ge0$,
\begin{equation}\label{eq:fractional-constant}
 (I_{T-}^{m\alpha}{\bf1})(t)
 =\frac{(T-t)^{m\alpha}}{\Gamma(m\alpha+1)}.
\end{equation}
The Mittag--Leffler function is
\begin{equation}\label{eq:mittag-leffler}
 E_\alpha(z):=\sum_{m=0}^\infty\frac{z^m}{\Gamma(m\alpha+1)}.
\end{equation}
See \cite{Diethelm2010,Podlubny1999Mittag,GripenbergLondenStaffans1990} for the
fractional-integral and Volterra background.
For positive $z$ and $0<\alpha<2$,
$E_\alpha(z)=\alpha^{-1}e^{z^{1/\alpha}}+O(z^{-1})$; also
$E_1(z)=e^z$ and $E_2(z)=\cosh\sqrt z$.
For the block argument we also use
$\log E_\alpha(z)\sim z^{1/\alpha}$ on the positive axis for every $\alpha>0$.
A Stirling--Laplace proof of this all-orders relation is recorded in the supplement.

\begin{lemma}[Gibbs differential and global Lipschitz bound]
\label{lem:gibbs-differential}
For $q,h\in L^\infty(\mu)$ and $\tau>0$,
\begin{equation}\label{eq:gibbs-derivative}
 D\cG_\tau(q)[h](da)
 =\frac{\cG_\tau(q)(da)}{\tau}
  \left(h(a)-\int h\,d\cG_\tau(q)\right).
\end{equation}
Moreover, for all $q,\tilde q$,
\begin{equation}\label{eq:gibbs-global-lipschitz}
 \norm{\cG_\tau(q)-\cG_\tau(\tilde q)}_{\rm TV}
 \le \min\left\{2,\frac{\osc(q-\tilde q)}{\tau}\right\},
 \qquad
 \osc f:=\operatorname*{ess\,sup}_{\mu}f
              -\operatorname*{ess\,inf}_{\mu}f.
\end{equation}
Here $D\cG_\tau(q)$ is the Fr\'echet derivative into the finite signed measures
equipped with the total-variation norm.
\end{lemma}

\begin{proof}
Differentiating the normalized exponential gives \eqref{eq:gibbs-derivative}.  For the
global bound, interpolate $q_u=\tilde q+u(q-\tilde q)$, integrate the derivative from
$u=0$ to one, and use
$\int|h-\int h\,d\rho|d\rho\le\osc h$.  The bound by two is the diameter of the
probability simplex in our convention.
\end{proof}

At a reference-tied model $M^\circ$, the equilibrium policy equals a reference
$\mu$ and every acting-self score is constant over actions.  For $h\in\mathsf H$,
let $c[h]\in\cY$ be the direct score derivative, with the future policy held fixed.
Thus $c:\mathsf H\to\cY$ is a bounded linear map.
Let $\cK:\cX\to\cY$ be the future-policy-to-current-score derivative and let
$\cS_t:Y\to X$ be the centered covariance operator in \eqref{eq:gibbs-derivative} at the
reference policy.  Its bounded pointwise lift $\cS:\cY\to\cX$ is
$(\cS y)(t)=\cS_t y(t)$.  The equilibrium tangent
$u=D_M\Phi_\tau(M^\circ)[h]$ formally
satisfies
\begin{equation}\label{eq:abstract-tangent}
 u(t)=\frac1\tau\cS_t\{c[h](t)+(\cK u)(t)\}.
\end{equation}

\begin{assumption}[Fractional causal influence]\label{ass:fractional-causal}
There are $\alpha\in(0,2]$, $\kappa\ge0$, and $C_c<\infty$ such that, for every
$v\in\cX$, every $h\in\mathsf H$, and every $t\in[0,T]$,
\begin{align}
 \norm{\cS_t(\cK v)(t)}_X
 &\le \kappa(I_{T-}^\alpha\norm{v(\cdot)}_X)(t),
                                                        \label{eq:fractional-kernel-bound}\\
 \sup_t\norm{\cS_tc[h](t)}_X&\le C_c\norm{h}_{\mathsf H}. \label{eq:direct-bound}
\end{align}
\end{assumption}

For a bounded Volterra kernel, $\alpha=1$.  A kernel that vanishes linearly on the
diagonal has $\alpha=2$.  The $D_x$ estimate for a uniformly elliptic parabolic
semigroup has the weak singularity $(s-t)^{-1/2}$, hence $\alpha=1/2$.

\begin{theorem}[Fractional Volterra susceptibility]\label{thm:fractional-resolvent}
Under \cref{ass:fractional-causal}, \eqref{eq:abstract-tangent} has a unique solution
for every $\tau>0$, given by the convergent causal resolvent
\begin{equation}\label{eq:volterra-neumann}
 u=\frac1\tau\sum_{m=0}^\infty
       \tau^{-m}(\cS\cK)^m\cS c[h].
\end{equation}
For $0\le t\le T$,
\begin{equation}\label{eq:mittag-bound}
 \norm{u(t)}_X
 \le\frac{C_c\norm{h}_{\mathsf H}}{\tau}
 E_\alpha\!\left(\frac{\kappa(T-t)^\alpha}{\tau}\right).
\end{equation}
In particular, as $\tau\downarrow0$, the upper bound has exponential scale
\begin{equation}\label{eq:mittag-asymptotic-scale}
 \frac1\tau\exp\left\{
   \kappa^{1/\alpha}(T-t)\tau^{-1/\alpha}\right\}.
\end{equation}
\end{theorem}

\begin{proof}
Write $A=\cS\cK$.  Repeated use of
\eqref{eq:fractional-kernel-bound}, the semigroup identity for fractional integrals,
and \eqref{eq:fractional-constant} gives
\[
 \norm{A^m\cS c[h](t)}_X
 \le C_c\norm{h}_{\mathsf H}\,
 \frac{\kappa^m(T-t)^{m\alpha}}{\Gamma(m\alpha+1)}.
\]
The majorant series is finite for every $\tau>0$ and sums to
\eqref{eq:mittag-bound}; hence \eqref{eq:volterra-neumann} converges absolutely in
$\cX$ and solves \eqref{eq:abstract-tangent}.  If $z=\tau^{-1}Az$, then for every
$m\ge1$,
\[
 \|z(t)\|_X\le \|z\|_\infty
 \frac{(\kappa/\tau)^m(T-t)^{m\alpha}}{\Gamma(m\alpha+1)}.
\]
The right-hand side tends to zero, which proves uniqueness.  When $\kappa>0$ and
$t<T$, the positive-axis Mittag--Leffler asymptotic yields
\eqref{eq:mittag-asymptotic-scale}.
\end{proof}

The norm estimate in \cref{thm:fractional-resolvent} need not be sharp without a sign
condition.  For example, the scalar operator
\[
 (Av)(t)=-\kappa\int_t^T v(s)\,ds
\]
satisfies the order-one kernel bound, but the response to unit forcing is
$\tau^{-1}e^{-\kappa(T-t)/\tau}$.  Negative feedback damps the perturbation.  A
matching lower bound requires a direction in which the causal feedback keeps its sign.

Let $X_+\subset X$ be a closed convex pointed cone.  Write $x\preceq y$ when
$y-x\in X_+$ and use the pointwise order on $\cX$.  For policy tangents, $X_+$ may be
the ray generated by one action contrast rather than the usual order on signed
measures.

\begin{theorem}[Aligned-mode lower bound]
\label{thm:aligned-mode-lower}
Suppose the conditions of \cref{ass:fractional-causal} hold.  Set
\[
 A=\cS\cK,\qquad w_h(t)=\cS_tc[h](t).
\]
Suppose that for one oriented direction $h\in\mathsf H\setminus\{0\}$ there are constants $c_->0$ and
$0<\kappa_-\le\kappa$, a path $\psi\in C([0,T];X_+)$, and positive functionals
$\ell_t\in X^*$ with $\norm{\ell_t}_{X^*}\le1$ and
$q(t):=\ell_t(\psi(t))>0$.  Assume that these objects do not depend on $\tau$ and that
\begin{align}
 A\cX_+&\subset\cX_+,                                      \label{eq:A-positive}\\
 w_h(t)&\succeq c_-\norm{h}_{\mathsf H}\,\psi(t),          \label{eq:forcing-aligned}\\
 A(f\psi)(t)&\succeq
 \kappa_-\bigl(I_{T-}^{\alpha}f\bigr)(t)\psi(t)             \label{eq:mode-subeigen}
\end{align}
for every nonnegative $f\in C([0,T])$.  Then
\begin{equation}\label{eq:matched-mittag-bound}
 \frac{c_-q(t)\norm{h}_{\mathsf H}}{\tau}
 E_\alpha\!\left(\frac{\kappa_-(T-t)^\alpha}{\tau}\right)
 \le \norm{u(t)}_X
 \le
 \frac{C_c\norm{h}_{\mathsf H}}{\tau}
 E_\alpha\!\left(\frac{\kappa(T-t)^\alpha}{\tau}\right).
\end{equation}
For each fixed $t<T$,
\begin{align}
 \kappa_-^{1/\alpha}(T-t)
 &\le
 \liminf_{\tau\downarrow0}\tau^{1/\alpha}
 \log\frac{\tau\norm{u(t)}_X}{\norm{h}_{\mathsf H}} \notag\\
 &\le
 \limsup_{\tau\downarrow0}\tau^{1/\alpha}
 \log\frac{\tau\norm{u(t)}_X}{\norm{h}_{\mathsf H}}
 \le \kappa^{1/\alpha}(T-t).                              \label{eq:matched-log-rate}
\end{align}
If $\kappa_-=\kappa$, the upper exponential rate is attained.
\end{theorem}

\begin{proof}
Let $f_m=I_{T-}^{m\alpha}{\bf1}$.  Induction gives
\begin{equation}\label{eq:positive-iterate-lower}
 A^mw_h(t)\succeq
 c_-\norm{h}_{\mathsf H}\,\kappa_-^m f_m(t)\psi(t),\qquad m\ge0.
\end{equation}
The case $m=0$ is \eqref{eq:forcing-aligned}.  If it holds at $m$, positivity of $A$
and \eqref{eq:mode-subeigen} give
\[
 A^{m+1}w_h
 \succeq c_-\norm{h}_{\mathsf H}\,\kappa_-^{m+1}
 I_{T-}^{\alpha}f_m\,\psi
 =c_-\norm{h}_{\mathsf H}\,\kappa_-^{m+1}f_{m+1}\psi.
\]
Insert \eqref{eq:positive-iterate-lower} into \eqref{eq:volterra-neumann}.  Both series
converge uniformly and the cone is closed, hence
\[
 u(t)\succeq
 \frac{c_-\norm{h}_{\mathsf H}}{\tau}\psi(t)
 E_\alpha\!\left(\frac{\kappa_-(T-t)^\alpha}{\tau}\right).
\]
Apply $\ell_t$ and use $\norm{\ell_t}\le1$.  This proves the lower bound; the upper
bound is \eqref{eq:mittag-bound}.  The positive-axis Mittag--Leffler asymptotic gives
\eqref{eq:matched-log-rate}.
\end{proof}

\begin{corollary}[Fractional statistical resolution scale]
\label{cor:fractional-resolution}
Under the hypotheses of \cref{thm:aligned-mode-lower}, suppose
$\kappa_-=\kappa>0$.  Fix $t<T$, let
$h_n=n^{-1/2}h$, and denote its tangent response by $u_{\tau}[h_n]$.  Put
\[
 c_\alpha=2^\alpha\kappa(T-t)^\alpha.
\]
If $\tau_n(\log n)^\alpha\to c\in(0,\infty)$, then
\[
 \norm{u_{\tau_n}[h_n](t)}_X\longrightarrow
 \begin{cases}
  0,      &c>c_\alpha,\\
  \infty, &0<c<c_\alpha.
 \end{cases}
\]
Thus the root-$n$ linear resolution boundary has order $(\log n)^{-\alpha}$.
\end{corollary}

\begin{proof}
By \eqref{eq:matched-mittag-bound} and the positive-axis asymptotic,
\[
 \log\norm{u_{\tau_n}[h_n](t)}_X
 =-\frac12\log n+\kappa^{1/\alpha}(T-t)\tau_n^{-1/\alpha}
 +O(\log\log n).
\]
The coefficient of $\log n$ is negative for $c>c_\alpha$ and positive for
$c<c_\alpha$.
\end{proof}

For $\kappa>0$ and fixed $t<T$, consider the canonical balance
\[
 \tau^{-1}E_\alpha\!\left(\frac{\kappa(T-t)^\alpha}{\tau}\right)=\sqrt n,
\]
and set $a=\kappa^{1/\alpha}(T-t)$.  With $W$ as above, its asymptotic solution is
\begin{equation}\label{eq:fractional-critical-temperature}
 \bar\tau_n=
 \left[
 \frac{a}{\alpha W\!\left(a\alpha^{1/\alpha-1}n^{1/(2\alpha)}\right)}
 \right]^\alpha\{1+o(1)\}
 \sim\left(\frac{2a}{\log n}\right)^\alpha.
\end{equation}
This is a linearized resolution boundary.  A nonlinear phase law also requires control
of the small-temperature remainder; \cref{thm:root-n-phase} supplies that control in
the affine model.

To identify the relevant return channels, decompose
$X=X_1\times\cdots\times X_d$, write
$A:=\cS\cK\in\mathcal L(\cX)$ and
$\mathsf C:=\cS c\in\mathcal L(\mathsf H,\cX)$, and set $A=(A_{ij})$.
Give $A$ the directed graph with an edge $j\to i$ when
$A_{ij}\ne0$.  Directed walks index the blocks of $A^m$
\cite{BermanPlemmons1994}, allowing the Neumann series to distinguish finite causal paths
from repeated cyclic returns.

\begin{theorem}[Block feedback dichotomy]
\label{thm:block-feedback-dichotomy}
Let $D_\tau h$ denote the solution of
\begin{equation}\label{eq:block-tangent}
 D_\tau h=\frac1\tau\{\mathsf C h+A D_\tau h\}.
\end{equation}
Thus $D_\tau:\mathsf H\to\cX$.  The block operators $A$ and $\mathsf C$, their graph,
and every cone, vector, functional, and
constant in the hypotheses below are independent of $\tau$.
\begin{enumerate}[label=\textup{(\roman*)},leftmargin=2em]
\item If the block graph is acyclic with longest directed path $L$, then
$A^{L+1}=0$ and
\begin{equation}\label{eq:dag-block-resolvent}
 D_\tau=\frac1\tau\sum_{m=0}^{L}\tau^{-m}A^m\mathsf C.
\end{equation}
If $\mathsf C=0$, then $D_\tau=0$.  Otherwise, let
$m_\star=\max\{m\le L:A^m\mathsf C\ne0\}$.  Then
\begin{equation}\label{eq:dag-block-order}
 \|D_\tau\|_{\mathcal L(\mathsf H,\cX)}=\tau^{-(m_\star+1)}
 \{\|A^{m_\star}\mathsf C\|_{\mathcal L(\mathsf H,\cX)}+o(1)\}.
\end{equation}
In particular, the order is $\tau^{-(L+1)}$ when the forcing excites a nonzero
longest path.

\item Suppose each $X_i$ has a closed pointed cone, every block $A_{ij}$ is positive,
and the causal Neumann series for $\tau^{-1}A$ converges in $\cX$ for every
$\tau>0$.  Consider a directed cycle
\[
 i_0\to i_1\to\cdots\to i_{q-1}\to i_0
\]
with return operator
\[
 B=A_{i_0i_{q-1}}A_{i_{q-1}i_{q-2}}\cdots A_{i_1i_0}.
\]
Assume $\mathsf C h_\star$ belongs to the product cone for some
$h_\star\in\mathsf H\setminus\{0\}$.  Suppose there
are $c_\star,g,\beta>0$, a path
$\psi\in C([0,T];X_{i_0,+})$, and positive functionals $\ell_t$ of norm at most one,
such that, for every nonnegative $f\in C([0,T])$,
\begin{align}
 (\mathsf C h_\star)_{i_0}(t)&\succeq
 c_\star\|h_\star\|_{\mathsf H}\psi(t),              \label{eq:cycle-forcing}\\
 B(f\psi)(t)&\succeq
 g(I_{T-}^{\beta}f)(t)\psi(t),                        \label{eq:cycle-return}
\end{align}
and $\varrho(t):=\ell_t(\psi(t))>0$.  Then, for $t<T$,
\begin{equation}\label{eq:cycle-mittag-lower}
 \|(D_\tau h_\star)_{i_0}(t)\|
 \ge\frac{c_\star\varrho(t)\|h_\star\|_{\mathsf H}}{\tau}
 E_\beta\!\left(\frac{g(T-t)^\beta}{\tau^q}\right).
\end{equation}
Consequently,
\begin{equation}\label{eq:cycle-exponential-rate}
 \liminf_{\tau\downarrow0}\tau^{q/\beta}
 \log\{\tau\|(D_\tau h_\star)_{i_0}(t)\|\}
 \ge g^{1/\beta}(T-t).
\end{equation}
\end{enumerate}
\end{theorem}

\begin{proof}
The $(i,j)$ block of $A^m$ is a sum of operator products indexed by directed walks of
length $m$ from $j$ to $i$.  An acyclic graph has no walk longer than $L$, so
$A^{L+1}=0$ and the Neumann series gives \eqref{eq:dag-block-resolvent}.  Multiplying
by $\tau^{m_\star+1}$ leaves $A^{m_\star}\mathsf C$ plus terms that vanish in operator norm,
which proves \eqref{eq:dag-block-order}.

For the cycle, positivity implies that every Neumann term is nonnegative.  The
$(i_0,i_0)$ block of $A^q$ contains $B$; for every product-cone-valued $z$, the
remaining walk terms are positive and hence $(A^qz)_{i_0}\succeq Bz_{i_0}$.
Thus induction and
\eqref{eq:cycle-forcing}--\eqref{eq:cycle-return} give
\[
 (A^{qm}\mathsf C h_\star)_{i_0}
 \succeq c_\star\|h_\star\|_{\mathsf H}g^m
 I_{T-}^{m\beta}{\bf1}\,\psi .
\]
Keeping the terms indexed by $qm$ in the causal series and applying $\ell_t$ proves
\eqref{eq:cycle-mittag-lower}.  The logarithmic positive-axis asymptotic
$\log E_\beta(x)\sim x^{1/\beta}$, valid for every $\beta>0$, gives
\eqref{eq:cycle-exponential-rate}.
\end{proof}

Under the positivity hypotheses above, the return condition can be checked edge by
edge.  Let $i_q=i_0$, choose paths
$\psi_r\in C([0,T];X_{i_r,+})$ with $\psi_q=\psi_0=: \psi$, and let
$a_r,\alpha_r>0$.  If, for every nonnegative $f\in C([0,T])$,
\begin{equation}\label{eq:edgewise-fractional-lower}
 A_{i_{r+1}i_r}(f\psi_r)
 \succeq a_r I_{T-}^{\alpha_r}f\,\psi_{r+1},
 \qquad r=0,\ldots,q-1,
\end{equation}
then \eqref{eq:cycle-return} holds with
$g=\prod_{r=0}^{q-1}a_r$ and $\beta=\sum_{r=0}^{q-1}\alpha_r$.

\begin{corollary}[Cycle resolution scale]
\label{cor:block-cycle-resolution}
Under the cyclic hypotheses of \cref{thm:block-feedback-dichotomy}, fix $t<T$, put
$H=T-t$, and let $h_n=n^{-1/2}h_\star$.  If
$\tau_n(\log n)^{\beta/q}\to\vartheta\in(0,\infty)$, then
\begin{equation}\label{eq:cycle-resolution-lower}
 \liminf_{n\to\infty}
 \frac{\log\|(D_{\tau_n}h_n)_{i_0}(t)\|}{\log n}
 \ge-\frac12+g^{1/\beta}H\vartheta^{-q/\beta}.
\end{equation}
Thus the response diverges when
\begin{equation}\label{eq:cycle-resolution-critical-lower}
 \vartheta<2^{\beta/q}g^{1/q}H^{\beta/q}.
\end{equation}
If, in addition, \cref{ass:fractional-causal} holds with
$\alpha=\beta/q\in(0,2]$ and $g=\kappa^q$, then the matching upper bound shows that
the response vanishes when the inequality in
\eqref{eq:cycle-resolution-critical-lower} is reversed.  The root-$n$ linear-response
boundary is therefore $(\log n)^{-\beta/q}$.
\end{corollary}

\begin{proof}
Apply the logarithmic Mittag--Leffler asymptotic to
\eqref{eq:cycle-mittag-lower}; the exterior factor contributes
$\log(\tau_n^{-1})=O(\log\log n)=o(\log n)$.  The upper half follows from
\eqref{eq:mittag-bound} with $\alpha=\beta/q$.
\end{proof}

The ratio
\begin{equation}\label{eq:cycle-mean-order}
 \bar\alpha_C=\frac1{q_C}\sum_{e\in C}\alpha_e
\end{equation}
is the mean causal smoothing order of a positive cycle $C$.  Among several cycles,
the smallest $\bar\alpha_C$ gives the strongest power of $1/\tau$ in the exponential;
ties are broken by the geometric return gain.  Among the lower bounds certified by
\eqref{eq:edgewise-fractional-lower}, suppose a mode $\psi_i$ is fixed at each
vertex and every edge bound uses these same vertex modes.  Finding the strongest
guaranteed exponent is then a minimum-mean-cycle problem in the sense of
\cite{Karp1978}.  This statement does not
exclude another operator component from producing a larger full-resolvent norm.  For
edge kernels bounded below by positive constants along the cycle, $\alpha_e=1$, and
the cycle guarantees inverse-log linear-response divergence below its cycle constant;
a matching upper bound identifies the full boundary.

The same causal iteration gives a nonlinear bound for model and numerical error.  For
two policies define
$d(t)=\sup_{x}\norm{\pi(t,x)-\tilde\pi(t,x)}_{\rm TV}$.

\begin{theorem}[Nonlinear model-and-residual stability]
\label{thm:nonlinear-stability}
Let $\pi$ be an exact causal Gibbs equilibrium for $M$.  Let
$\eta\in L^\infty(0,T)$ be nonnegative, and let $\tilde\pi$ have fixed-point
residual at most $\eta(t)$ for $\widetilde M$:
\begin{equation}\label{eq:fixed-point-residual}
 \sup_x\norm{\tilde\pi(t,x)-
 \cG_\tau(\cQ_{\widetilde M}(\tilde\pi)(t,x,\cdot))}_{\rm TV}
 \le\eta(t).
\end{equation}
Suppose the score maps obey the following bound pairwise for every two candidates in
the class, for some $\alpha>0$, $L_M,\kappa\ge0$, and primitive distance
$\varepsilon=d(M,\widetilde M)$:
\begin{equation}\label{eq:nonlinear-score-bound}
 \sup_x\osc\{\cQ_M(\pi)-\cQ_{\widetilde M}(\tilde\pi)\}(t,x,\cdot)
 \le L_M\varepsilon+\kappa(I_{T-}^\alpha d)(t).
\end{equation}
Then
\begin{equation}\label{eq:nonlinear-stability-bound}
 d(t)\le
 \left(\norm{\eta}_\infty+\frac{L_M\varepsilon}{\tau}\right)
 E_\alpha\!\left(\frac{\kappa(T-t)^\alpha}{\tau}\right).
\end{equation}
In particular, exact equilibria are unique within any class satisfying
\eqref{eq:nonlinear-score-bound} with $\varepsilon=0$.
\end{theorem}

\begin{remark}
The constants $L_M$ and $\kappa$ are understood to be uniform over any temperature
range for which the displayed dependence is used.  The quantity $\eta$ is a policy
fixed-point residual.  A sup-norm score residual $\rho$ first passes through the Gibbs
map and contributes at most $2\rho/\tau$ to $\eta$ (or $\rho/\tau$ when $\rho$ denotes
score oscillation).
\end{remark}

\begin{proof}
The triangle inequality, \eqref{eq:fixed-point-residual}, and
\cref{lem:gibbs-differential} give
\[
 d(t)\le\eta(t)+\frac{L_M\varepsilon}{\tau}
       +\frac\kappa\tau(I_{T-}^\alpha d)(t).
\]
Put $a=\|\eta\|_\infty+L_M\varepsilon/\tau$.  Since $d\le2$, $N$ iterations give
\[
 d(t)\le a\sum_{m=0}^{N-1}
 \frac{[\kappa(T-t)^\alpha/\tau]^m}{\Gamma(m\alpha+1)}
 +2\frac{[\kappa(T-t)^\alpha/\tau]^N}{\Gamma(N\alpha+1)}.
\]
The last term tends to zero as $N\to\infty$, which proves
\eqref{eq:nonlinear-stability-bound}.  If both the model discrepancy and the
residual vanish, the same estimate gives $d(t)=0$ for every $t$.
\end{proof}

For exact equilibria, so that $\eta_n=0$, if
$\varepsilon_n=O_{\Pp}(n^{-1/2})$, a sufficient condition for the model-error bound
in \eqref{eq:nonlinear-stability-bound} to vanish in probability is
\begin{equation}\label{eq:sufficient-annealing}
 \frac{1}{\sqrt n\,\tau_n}
 E_\alpha\!\left(\frac{\kappa T^\alpha}{\tau_n}\right)\longrightarrow0.
\end{equation}
For an approximate fixed point one must additionally require
\[
 \|\eta_n\|_\infty
 E_\alpha\!\left(\frac{\kappa T^\alpha}{\tau_n}\right)
 \longrightarrow0
\]
in probability, or deterministically when the residual is deterministic.
When $\kappa>0$, the exponential term gives the critical order
$(\log n)^{-\alpha}$.  More precisely, a sufficient condition is
\begin{equation}\label{eq:fractional-log-barrier}
 \liminf_{n\to\infty}\tau_n(\log n)^\alpha
 >2^\alpha\kappa T^\alpha.
\end{equation}
If $\kappa=0$, there is no logarithmic barrier and
$(\sqrt n\,\tau_n)^{-1}\to0$ is the corresponding sufficient condition.
This is only a sufficient general law: stable action gaps can make the Gibbs covariance
shrink with $\tau$, whereas tied modes can attain the full amplification.  The exact
construction in \cref{sec:exact} proves a sharp nonlinear law for $\alpha=1$.

\section{EEHJB stability and influence equations}
\label{sec:stability}

This section turns the abstract causal mechanism into two EEHJB results.  First, it
derives a temperature-explicit stability estimate for learned-model perturbations.
Second, at fixed positive temperature, it constructs a twice differentiable local
equilibrium branch and its influence equations.  The weak space below is sufficient for
the mild fixed-point and stability arguments; the stronger H\"older space is used only
to recover classical solutions.  A separate comparison envelope records the additional
uniformity needed when the temperature also varies.

\subsection{Temperature-explicit stability}

We work directly with the following diffusion system.  Let
$r=r(y,t,x,a)$, let $\delta$ be a fixed bounded $C^1$ discount kernel
with $\delta(0)=1$, and let $G(\vartheta,y,z)$ be fixed and smooth.  Its first two
arguments are reference variables, not flow-state variables.  For a policy $\pi$, set
\begin{align}
 B^\pi(t,x)&=\int_\A b(t,x,a)\pi(t,x,da),&
 R_y^\pi(t,x)&=\int_\A r(y,t,x,a)\pi(t,x,da).
                                                        \label{eq:general-policy-averages}
\end{align}
Its policy-evaluation pair is the solution of
\begin{equation}\label{eq:general-V1}
\begin{aligned}
0={}&\partial_tV^{1,\pi}+\tfrac12\operatorname{tr}
 (\sigma\sigma^\top D_x^2V^{1,\pi})+B^\pi\cdot D_xV^{1,\pi}\\
&+\delta(t-\vartheta)\{R_y^\pi-\tau\KL(\pi\Vert\mu)\},
\qquad V^{1,\pi}(\vartheta,T,y,x)=F(\vartheta,y,x).
\end{aligned}
\end{equation}
\begin{equation}\label{eq:general-V2}
\begin{aligned}
0={}&\partial_tV^{2,\pi}+\tfrac12\operatorname{tr}
 (\sigma\sigma^\top D_x^2V^{2,\pi})+B^\pi\cdot D_xV^{2,\pi},\\
&\hspace{7em}V^{2,\pi}(T,x)=h(x).
\end{aligned}
\end{equation}
The variables $(\vartheta,y)$ are frozen in \eqref{eq:general-V1} and in $G$.  The
acting self's criterion is
\[
 J(t,x;\pi):=V^{1,\pi}(t,t,x,x)+G(t,x,V^{2,\pi}(t,x)).
\]
Its policy-relevant flow-state gradient is
\begin{equation}\label{eq:Z-general}
 Z_V(t,x)=D_xV^1(t,t,x,x)
 +G_z(t,x,V^2(t,x))D_xV^2(t,x).
\end{equation}
For a model $M=(b,r,F,h)$, put
\begin{align}
 q_{M,V}(t,x,a)&=b(t,x,a)\cdot Z_V(t,x)+r(x,t,x,a),
                                                        \label{eq:q-general}\\
 \pi_{M,V}(t,x)&=\cG_\tau(q_{M,V}(t,x,\cdot)).          \label{eq:pi-general}
\end{align}
An equilibrium is a classical pair $V=(V^1,V^2)$ satisfying
\eqref{eq:general-V1}--\eqref{eq:general-V2} under $\pi=\pi_{M,V}$.  This is the full
system needed below.  The spike calculation in \cref{prop:spike-verification} applies
with the extra flow-state gradient $G_zD_xV^2$ already included in $Z_V$.  The first
two arguments of $G$ remain frozen, so no reference-state derivative enters.  Thus the
PDE fixed point is equivalent to the local equilibrium condition.

Stability and the implicit-function argument use the weak value space
\[
 \mathcal D_1=\{(\vartheta,t,y,x):0\leq\vartheta\leq t\leq T,
                  \ y,x\in\R^d\},
 \qquad \mathcal Q=[0,T]\times\R^d.
\]
Let $\mathfrak V$ be the product space of pairs $V=(V^1,V^2)$ for which
$V^1,D_xV^1$ are bounded and uniformly continuous on $\mathcal D_1$, and
$V^2,D_xV^2$ are bounded and uniformly continuous on $\mathcal Q$.  Its norm is
\begin{align}
 \|V\|_{\mathfrak V}:={}&
 \sup_{\mathcal D_1}(|V^1|+|D_xV^1|)
 +\sup_{\mathcal Q}(|V^2|+|D_xV^2|).                 \label{eq:X1-norm}
\end{align}
For $t_0\in[0,T]$, its tail norm is
\begin{align}
 \|V\|_{\mathfrak V,[t_0,T]}:={}&
 \sup_{\substack{(\vartheta,t,y,x)\in\mathcal D_1\\t\ge t_0}}
 (|V^1|+|D_xV^1|)
 +\sup_{\substack{(t,x)\in\mathcal Q\\t\ge t_0}}
 (|V^2|+|D_xV^2|).                                  \label{eq:X1-tail-norm}
\end{align}
The derivatives are classical flow-state derivatives.  Uniform convergence of a
function and its derivative preserves this relation along line segments, so
$\mathfrak V$ is complete.  The diagonal flow-gradient map
\begin{equation}\label{eq:diagonal-trace-weak}
 \mathsf T V^1(t,x):=D_xV^1(t,t,x,x)
\end{equation}
has operator norm at most one from the first component of $\mathfrak V$ into
$C_b(\mathcal Q;\R^d)$.  It is the flow-state derivative restricted to the diagonal,
not the total derivative of $x\mapsto V^1(t,t,x,x)$.

The coefficient norm is defined next.  If $E$ contains a distinguished
flow-state variable $x$, let $\mathcal B_x^1(E)$ consist of functions for which
$f$ and $D_xf$ are bounded and uniformly continuous in all variables, including the
action variable when present, with norm
$\|f\|_{\mathcal B_x^1}=\sup_E(|f|+|D_xf|)$.  Vector-valued versions use the
Euclidean norm.  Set
\begin{align}
 \mathfrak C^1:={}&
 \mathcal B_x^1(\mathcal Q\times\A;\R^d)
 \times\mathcal B_x^1(\R^d\times\mathcal Q\times\A)
 \times\mathcal B_x^1([0,T]\times\R^d\times\R^d)
 \times\mathcal B_x^1(\R^d),                          \label{eq:coefficient-space}
\end{align}
for $(b,r,F,h)$ in that order.  Supremum norms over reference variables and actions
are part of this definition.

Classical regularity uses a stronger space.  Put
$\rho_r=|y-y'|+|t-t'|^{1/2}+|x-x'|$,
$\rho_{\rm ref}=|\vartheta-\vartheta'|^{1/2}+|y-y'|$, and
$F^\sharp(\vartheta,y)=F(\vartheta,y,\cdot)$.  The space
$\mathfrak C^{\rm cl}_\gamma\subset\mathfrak C^1$ requires, uniformly in $a$,
$b\in C_b^{\gamma/2,1+\gamma}$, $r,D_xr$ to be $\gamma$-H\"older in $\rho_r$,
$F^\sharp$ to be bounded in $C_b^{2+\gamma}$ and $\gamma$-H\"older in
$\rho_{\rm ref}$, and $h\in C_b^{2+\gamma}$.  Its exact norm is in the supplement.

The hypotheses needed for the fixed-temperature calculus are separated from the
additional envelope needed for uniform small-temperature comparison.

\begin{assumption}[Data and diffusion evolution]\label{ass:eehjb-data}
Fix $\gamma\in(0,1)$.  The action space is compact metric and $\mu$ has full support.  Put
$a=\sigma\sigma^\top$ and
$\mathcal L_t^0=\tfrac12\operatorname{tr}(a(t,\cdot)D_x^2)$.  The following hold on
the parameter neighborhood under consideration.
\begin{enumerate}[label=\textup{(\roman*)},leftmargin=2em]
\item The uncontrolled matrix $a$ is bounded and uniformly elliptic.  It belongs to
$C_b^{\gamma/2,1+\gamma}(\mathcal Q)$, and its evolution family $P^0_{t,s}$ satisfies
\begin{align}
 \|P^0_{t,s}f\|_\infty&\le C_0\|f\|_\infty,\notag\\
 \|D_xP^0_{t,s}f\|_\infty
   &\le C_0(s-t)^{-1/2}\|f\|_\infty,\notag\\
 \|P^0_{t,s}g\|_{\mathcal B_x^1}&\le C_0\|g\|_{\mathcal B_x^1}.
                                                               \label{eq:base-semigroup}
\end{align}
For each $\eta\in(0,\gamma)$ and $0\le t\le t'<s\le T$, it also satisfies
\begin{align}
 [D_xP^0_{t,s}f]_{C_x^\eta}
 &\le C_\eta(s-t)^{-(1+\eta)/2}\|f\|_\infty,\notag\\
 \|D_xP^0_{t,s}f-D_xP^0_{t',s}f\|_\infty
 &\le C_\eta|t-t'|^{\eta/2}(s-t')^{-(1+\eta)/2}\|f\|_\infty.
                                                               \label{eq:base-holder}
\end{align}
For the classical conclusion, we also use the standard whole-space linear Schauder
estimates for this evolution family under the same coefficient assumptions.
\item The coefficient quadruple belongs to $\mathfrak C^1$.  For a
finite-dimensional family, $\theta\mapsto M_\theta$ is $C^2$ from an open subset of
$\R^p$ into $\mathfrak C^1$.
\item The discount kernel is $C^1$.  For every $R<\infty$ and $0\le j\le3$,
$\partial_z^jG$ is bounded and uniformly continuous on
$\mathcal Q\times[-R,R]$.  In particular,
\[
 \omega_{G,R}(\varepsilon):=
 \sup_{\substack{(t,x)\in\mathcal Q,\ |z|\vee|z'|\le R\\|z-z'|\le\varepsilon}}
 |G_{zzz}(t,x,z)-G_{zzz}(t,x,z')|\longrightarrow0.
\]
\item For the classical conclusion, $M_\theta$ and its first two parameter
derivatives are locally bounded in $\mathfrak C^{\rm cl}_\gamma$, and, on bounded
$z$-ranges, $\partial_z^jG$, $0\le j\le3$, are uniformly bounded in
$C_b^{\gamma/2,\gamma}(\mathcal Q)$.
\end{enumerate}
Items \textup{(i)}--\textup{(iii)} are the mild clauses; item \textup{(iv)} is the
classical add-on.  All bounds are local and uniform in the stated parameter
neighborhood.
\end{assumption}

\begin{assumption}[Uniform comparison envelope]\label{ass:eehjb-envelope}
Let $\mathcal I\subset(0,\bar\tau]$.  The selected classical equilibria compared in
\cref{thm:eehjb-stability} obey
\[
 \sup_{\tau\in\mathcal I}
 \bigl(\|V_{M,\tau}^\star\|_{\mathfrak V}
      +\|V_{\widetilde M,\tau}^\star\|_{\mathfrak V}\bigr)\le B_V,
\]
and the mild-clause bounds in \cref{ass:eehjb-data} are uniform over the same
collection.  Existence is asserted only for these selected solutions.  No evaluation
on an open value-function ball is included in this assumption.
\end{assumption}

Taking $\mathcal I=\{\tau\}$ gives a fixed-temperature comparison.  If
$0\in\overline{\mathcal I}$, the common bound is an extra small-temperature
hypothesis; fixed-temperature well-posedness does not supply it.

For two models with the same $(\sigma,G)$, use the primitive distance
\begin{equation}
 d_1(M,\widetilde M)=\|M-\widetilde M\|_{\mathfrak C^1}.
                                                               \label{eq:model-distance}
\end{equation}
For $t_0\in[0,T]$, set
\begin{equation}\label{eq:policy-tail-distance}
 d_{{\rm pol},t_0}(\pi,\widetilde\pi)
 :=\sup_{\substack{t\in[t_0,T]\\x\in\R^d}}
 \norm{\pi(t,x)-\widetilde\pi(t,x)}_{\rm TV},
 \qquad d_{\rm pol}(\pi,\widetilde\pi)
 :=d_{{\rm pol},0}(\pi,\widetilde\pi).
\end{equation}

For a candidate $V$ and model $M$, abbreviate
\begin{align}
 H_\tau(q)&:=\tau\log\int_\A e^{q(a)/\tau}\mu(da),\notag\\
 B_{M,V}&:=B^{\pi_{M,V}},&
 R_{M,y,V}&:=R_y^{\pi_{M,V}},\notag\\
 \Phi^1_{M,V}
 &:=\delta(t-\vartheta)\{H_\tau(q_{M,V})
   +R_{M,y,V}-R_{M,x,V}-B_{M,V}\cdot Z_V\}.
                                                        \label{eq:cancelled-source}
\end{align}
Here $R_{M,x,V}$ means that the frozen reward state is set equal to the current
state.  The Gibbs log-partition identity rewrites the first evaluation equation with
drift $B_{M,V}$ and source $\Phi^1_{M,V}$; its short derivation is recorded in the
supplement.

The next estimate uses the uncontrolled diffusion evolution and requires no spatial
derivative bound on the Gibbs-induced drift.

\begin{lemma}[Bounded-drift parabolic estimate]\label{lem:parabolic-gradient}
Suppose \cref{ass:eehjb-data}(i) holds.  Let $u$ be the mild solution of
\[
 \partial_tu+\mathcal L_t^0u+c\cdot D_xu+f=0,\qquad u(T)=g,
\]
where $\|c\|_\infty\le B_c$, $f$ is bounded, and $g\in\mathcal B_x^1$.  Then
\begin{equation}\label{eq:parabolic-gradient-estimate}
 \norm{u(t)}_\infty+\norm{D_xu(t)}_\infty
 \le C\norm{g}_{\mathcal B_x^1}
 +C\int_t^T\{1+(s-t)^{-1/2}\}\norm{f(s)}_\infty\,ds,
\end{equation}
where $C$ depends on $(C_0,B_c,T)$ and not on derivatives of $c$.  The same estimate
holds for a difference equation with a bounded drift forcing $\dot c\cdot D_xu$.
\end{lemma}

\begin{proof}
Duhamel's formula and \eqref{eq:base-semigroup}, with
$e(t)=\|u(t)\|_\infty+\|D_xu(t)\|_\infty$ and $k(r)=1+r^{-1/2}$, give
\[
 e(t)\le C\|g\|_{\mathcal B_x^1}
 +C\int_t^Tk(s-t)\|f(s)\|_\infty\,ds
 +CB_c\int_t^Tk(s-t)e(s)\,ds .
\]
Since $k(r)\le C_Tr^{-1/2}$, the $m$th resolvent iterate is bounded by
$(CB_c)^mT^{m/2}/\Gamma(m/2+1)$.  Summation proves the estimate and uniqueness;
short-interval contraction and continuation give existence.  For induced drifts,
$B_c=\sup_a\|b(\cdot,a)\|_\infty$, with no temperature factor.  The supplement
records the full resolvent closure.
\end{proof}

\begin{theorem}[Temperature-explicit learned-model stability]
\label{thm:eehjb-stability}
Let $M$ and $\widetilde M$ satisfy the mild clauses of
\cref{ass:eehjb-data,ass:eehjb-envelope}, with equilibrium pairs
$V_M^\star,V_{\widetilde M}^\star$ and policies $\pi_M^\star,
\pi_{\widetilde M}^\star$.  For every $\tau\in\mathcal I$,
\begin{align}
 \norm{V_M^\star-V_{\widetilde M}^\star}_{\mathfrak V}
 &\le S_\tau d_1(M,\widetilde M),                     \label{eq:value-stability}\\
 d_{\rm pol}(\pi_M^\star,\pi_{\widetilde M}^\star)
 &\le \min\left\{2,\frac{C(1+S_\tau)}{\tau}
               d_1(M,\widetilde M)\right\},          \label{eq:policy-stability}
\end{align}
where one may take
\begin{equation}\label{eq:S-tau}
 S_\tau=C E_{1/2}\!\left(C(1+\tau^{-1})T^{1/2}\right)
 \le C\exp\{CT(1+\tau^{-2})\}.
\end{equation}
Along a sequence $\tau\downarrow0$, this estimate requires
$\tau\in\mathcal I$ and the uniform envelope in \cref{ass:eehjb-envelope}.
For a fixed model, any two classical solutions in the envelope coincide; hence the
selected solution is unique in that class.
\end{theorem}

\begin{proof}
Let
$E(t)=\|V_M^\star-V_{\widetilde M}^\star\|_{\mathfrak V,[t,T]}$ and
$\varepsilon=d_1(M,\widetilde M)$.  Bounded derivatives of
$G$ give $\|Z_{V_M^\star}-Z_{V_{\widetilde M}^\star}\|_\infty\le CE(t)$, so the
score difference is bounded by $C\{E(t)+\varepsilon\}$ in oscillation.  Hence
\cref{lem:gibbs-differential} yields
\begin{equation}\label{eq:policy-difference-intermediate}
 d_{{\rm pol},t}(\pi_M^\star,\pi_{\widetilde M}^\star)\le
 \min\left\{2,C\tau^{-1}\{E(t)+\varepsilon\}\right\}.
\end{equation}
Write $\Delta$ for corresponding differences and take all suprema below over the tail
$[t,T]$.  Directly from the action averages, the log-partition identity, and the common
envelope,
\begin{align}
 \|\Delta B\|_\infty
 &\le C\{\varepsilon+d_{{\rm pol},t}
       (\pi_M^\star,\pi_{\widetilde M}^\star)\},\notag\\
 \|\Delta H_\tau\|_\infty&\le C\{\varepsilon+E(t)\},\notag\\
 \|\Delta\Phi^1\|_\infty
 &\le C\{\varepsilon+E(t)+d_{{\rm pol},t}
       (\pi_M^\star,\pi_{\widetilde M}^\star)\}.
                                                        \label{eq:source-differences}
\end{align}
Keep $B_{M,V_M^\star}$ as the principal drift.  The remaining drift forcing is
$\Delta B\cdot D_xV_{\widetilde M}^\star$ and is controlled by the first line of
\eqref{eq:source-differences}, without differentiating an induced drift.  Subtracting
the two cancelled systems, applying \cref{lem:parabolic-gradient}, and then using
\eqref{eq:policy-difference-intermediate} gives
\begin{equation}\label{eq:fractional-value-inequality}
 E(t)\le C\varepsilon+C(1+\tau^{-1})
 \int_t^T\{1+(s-t)^{-1/2}\}\{E(s)+\varepsilon\}\,ds.
\end{equation}
Since $E$ is nonincreasing, the pointwise estimate passes to its tail norm; on a finite
horizon the nonsingular kernel is dominated by the order-$1/2$ kernel.  The fractional
Gronwall iteration from
\cref{thm:fractional-resolvent} proves \eqref{eq:value-stability}--\eqref{eq:S-tau}.
Substitution in \eqref{eq:policy-difference-intermediate} proves
\eqref{eq:policy-stability}.
\end{proof}

At fixed temperature the factor in \eqref{eq:S-tau} can be absorbed into a constant.
If the learned model and temperature vary together within a common envelope, the
theorem gives the sufficient condition
\begin{equation}\label{eq:eehjb-sufficient-consistency}
 \varepsilon_n\tau_n^{-1}\exp(C/\tau_n^2)\longrightarrow0.
\end{equation}
This rate is a worst-case certificate, not a sharp universal boundary.  The exact model
in \cref{sec:exact} has a bounded Volterra kernel and a different, sharp
$e^{C/\tau}$ law.

\subsection{The EEHJB influence system and a delta method}

Let $\theta\in\Theta\subset\R^p$ parameterize $M_\theta=(b_\theta,r_\theta,
F_\theta,h_\theta)$ as in \cref{ass:eehjb-data}; $(\sigma,G)$ remain known.  Fix
$\tau>0$ and suppress
$(\theta,\tau)$ from the equilibrium notation.  Put
\begin{align}
 B(t,x)&=\int_\A b(t,x,a)\pi(t,x,da),\notag\\
 R_y(t,x)&=\int_\A r(y,t,x,a)\pi(t,x,da).             \label{eq:B-R}
\end{align}
For a direction $u\in\R^p$, dots denote derivatives at $\theta$ in direction $u$.
Define
\begin{align}
 \dot Z={}&D_x\dot V^1(t,t,x,x)
 +G_{zz}(t,x,V^2)[\dot V^2]D_xV^2
 +G_z(t,x,V^2)D_x\dot V^2,                            \label{eq:dot-Z}\\
 \dot q(a)={}&\dot b(a)\cdot Z+b(a)\cdot\dot Z
                   +\dot r(x,t,x,a),                   \label{eq:dot-q}\\
 \dot\pi(da)={}&\frac{\pi(da)}\tau
 \left\{\dot q(a)-\int_\A\dot q(a')\pi(da')\right\},   \label{eq:dot-pi}\\
 \dot B={}&\int_\A\dot b(a)\pi(da)+\int_\A b(a)\dot\pi(da),           \label{eq:dot-B}\\
 \dot R_y={}&\int_\A\dot r(y,t,x,a)\pi(da)
       +\int_\A r(y,t,x,a)\dot\pi(da).                 \label{eq:dot-R}
\end{align}

For the classical conclusion, let $\mathfrak V_\eta\subset\mathfrak V$ consist of
pairs with uniform $C_b^{1+\eta/2,2+\eta}$ flow regularity and bounded
$\eta$-H\"older dependence of $V^1,D_xV^1$ on the frozen variables, measured with
$|t-t'|^{1/2}+|x-x'|+|\vartheta-\vartheta'|^{1/2}+|y-y'|$.
The precise Banach norm is recorded in the supplement.  The only trace fact used below is
\begin{equation}\label{eq:diagonal-holder-trace}
 \|\mathsf T V^1\|_{C_b^{\eta/2,\eta}(\mathcal Q)}
 \le C\|V^1\|_{\mathfrak V^1_\eta},
 \qquad 0<\eta\le\gamma.
\end{equation}

Policy tangents are signed kernels, whereas the KL term is evaluated only on
probability kernels.  Let
\[
 \mathfrak P=C_b([0,T]\times\R^d;\mathcal M(\A))
\]
be the Banach space of finite signed kernels with norm
$\|\eta\|_{\mathfrak P}:=\sup_{t,x}\|\eta(t,x)\|_{\rm TV}$.
Policy derivatives lie in its closed zero-mass subspace $\mathfrak P_0$.

For a candidate $V\in\mathfrak V$, use $M=M_\theta$ in
\eqref{eq:cancelled-source}, set $\pi_{\theta,V}=\cG_\tau(q_{\theta,V})$, and define
$U=\mathcal E_{\theta,\tau}(V)$ by
\begin{align}
 U^1(\vartheta,t,y,\cdot)
 ={}&P^0_{t,T}F_\theta(\vartheta,y,\cdot)
 +\int_t^TP^0_{t,s}
 \{B_{\theta,V}\cdot D_xU^1+\Phi^1_{\theta,V}\}(s)\,ds,
                                                        \label{eq:mild-E1}\\
 U^2(t,\cdot)
 ={}&P^0_{t,T}h_\theta
 +\int_t^TP^0_{t,s}
 \{B_{\theta,V}(s)\cdot D_xU^2(s,\cdot)\}\,ds.        \label{eq:mild-E2}
\end{align}
Flow variables are suppressed inside the integrands.  Backward contraction and
continuation give a unique mild solution.  The Gibbs identity shows that the fixed
points of \eqref{eq:mild-E1}--\eqref{eq:mild-E2} are exactly those of the original
Gibbs-substituted system; they are classical whenever they lie in
$\mathfrak V_\eta$.  Define
\begin{equation}\label{eq:value-residual-map}
 \mathfrak F(\theta,V)
 :=V-\mathcal E_{\theta,\tau}(V).
\end{equation}

\begin{proposition}[Twice differentiable mild evaluation]
\label{prop:parabolic-calculus}
Fix $\tau>0$ and suppose the mild clauses of \cref{ass:eehjb-data} hold.  Locally on
every bounded open parameter-value neighborhood on which the stated coefficient bounds hold,
the mild evaluation map $\mathcal E_\tau$ is uniquely defined into $\mathfrak V$
and is twice continuously
Fr\'echet differentiable.  On each sufficiently small bounded convex neighborhood,
\begin{align}
 \|\mathcal E_\tau(z+\zeta)-\mathcal E_\tau(z)\|_{\mathfrak V}
 &\le C_\tau\|\zeta\|,\notag\\
 \|\mathcal E_\tau(z+\zeta)-\mathcal E_\tau(z)
       -D\mathcal E_\tau(z)[\zeta]\|_{\mathfrak V}
 &\le C_\tau\|\zeta\|^2,                              \label{eq:E-C2-main}\\
 \|D\mathcal E_\tau(z+\zeta)-D\mathcal E_\tau(z)\|_{\rm op}
 &\le C_\tau\|\zeta\|.\notag
\end{align}
For a pure value increment $W$, the terminal variation is zero and
\begin{equation}\label{eq:parabolic-calculus-causal}
 \|D_V\mathcal E_{\theta,\tau}(V)W\|_{\mathfrak V,[t,T]}
 \le C(1+\tau^{-1})
 I_{T-}^{1/2}\!\left(\|W\|_{\mathfrak V,[\,\cdot,T]}\right)(t).
\end{equation}
The constants in the differentiability statements are fixed-temperature constants.
\end{proposition}

\begin{proof}
The bounded diagonal trace and fixed-temperature Gibbs derivative bounds make the
induced drift and cancelled source $C^2$ in the relevant supremum norms.
Differentiating once and twice gives linear parabolic equations, to which
\cref{lem:parabolic-gradient} applies, proving \eqref{eq:E-C2-main}; the remainder
equations and continuity of the second derivative are in the supplement.  For a pure
value direction the terminal variation vanishes and the forcing at time $s$ is bounded
by $C(1+\tau^{-1})\|W\|_{\mathfrak V,[s,T]}$, which gives
\eqref{eq:parabolic-calculus-causal}.
\end{proof}

\begin{proposition}[Classical regularity of a mild fixed point]
\label{prop:mild-to-classical}
Suppose the mild clauses and the classical add-on of \cref{ass:eehjb-data} hold.  If
$V\in\mathfrak V$ satisfies $V=\mathcal E_{\theta,\tau}(V)$ for fixed $\tau>0$,
then $V\in\mathfrak V_\eta$ for every $\eta\in(0,\gamma)$.  Hence the mild fixed
point is a classical solution of the entropy-cancelled system and of the original
Gibbs-substituted EEHJB system.
\end{proposition}

The parabolic bootstrap and the reference-variable difference estimate are given in
the supplementary material.

\begin{theorem}[Forced tangent EEHJB and statistical delta method]
\label{thm:influence-clt}
Under the mild clauses of \cref{ass:eehjb-data}, let $\theta_0$ be interior to
$\Theta$ and suppose
$\theta\mapsto M_\theta$ is $C^2$ into the stated coefficient spaces on an open
neighborhood of $\theta_0$.  Fix $\tau>0$ and suppose $(V_0,\pi_0)$ is one mild
equilibrium at $\theta_0$.  Then there is a neighborhood $U$ of $\theta_0$ and a
unique local mild equilibrium branch
$\Psi_\tau(\theta):=(V^\star_{\theta,\tau},\pi^\star_{\theta,\tau})$ from $U$ into
$\mathfrak V\times\mathfrak P$, with
$\Psi_\tau(\theta_0)=(V_0,\pi_0)$.  Uniqueness is among zeros whose value lies in a
fixed $\mathfrak V$-neighborhood of $V_0$.  The branch is $C^2$ at fixed temperature.  Its
derivative is the unique mild solution of
\eqref{eq:dot-Z}--\eqref{eq:dot-R} and
\begin{align}
 0={}&\partial_t\dot V^1+\tfrac12\operatorname{tr}
 (\sigma\sigma^\top D_x^2\dot V^1)+B\cdot D_x\dot V^1
 +\dot B\cdot D_xV^1\notag\\
 &+\delta(t-\vartheta)\left[
   \dot R_y
   -\tau\int_\A\log\frac{d\pi}{d\mu}(a)\,\dot\pi(da)
 \right],
 \qquad \dot V^1(\vartheta,T,y,x)=\dot F(\vartheta,y,x),
                                                               \label{eq:tangent-V1}\\
 0={}&\partial_t\dot V^2+\tfrac12\operatorname{tr}
 (\sigma\sigma^\top D_x^2\dot V^2)+B\cdot D_x\dot V^2
 +\dot B\cdot D_xV^2,
 \qquad \dot V^2(T,x)=\dot h(x).                         \label{eq:tangent-V2}
\end{align}
Here $\delta$ is the fixed discount kernel; omit it when discounting has already been
absorbed into $r$.  Under the classical add-on of \cref{ass:eehjb-data}, every value
on this branch belongs to $\mathfrak V_\eta$, $0<\eta<\gamma$, and is classical.

More explicitly, there is a modulus $\omega_\tau(r)\downarrow0$ as $r\downarrow0$
such that
\begin{equation}\label{eq:frechet-remainder}
 \|\Psi_\tau(\theta_0+v)-\Psi_\tau(\theta_0)
     -D\Psi_\tau(\theta_0)[v]\|_{\mathfrak V\times\mathfrak P}
 \le \omega_\tau(\|v\|)\|v\|.
\end{equation}
No uniformity of $\omega_\tau$ as $\tau\downarrow0$ is asserted.

The statistical input below is the finite-dimensional parameter
$\theta\in\R^p$; the output is function valued.  We do not claim a nonparametric
coefficient-process delta method.

If an estimator satisfies
\begin{equation}\label{eq:theta-clt}
 \sqrt n(\widehat\theta_n-\theta_0)\Rightarrow\Xi\quad\text{in }\R^p,
\end{equation}
then $\Pp(\widehat\theta_n\in U)\to1$, and on this event
\begin{equation}\label{eq:equilibrium-functional-clt}
 \sqrt n\left\{(V^\star_{\widehat\theta_n,\tau},
 \pi^\star_{\widehat\theta_n,\tau})
 -(V^\star_{\theta_0,\tau},\pi^\star_{\theta_0,\tau})\right\}
 \Rightarrow(\dot V[\Xi],\dot\pi[\Xi]).
\end{equation}
The explicit local Gibbs factor in \eqref{eq:dot-pi} contributes $1/\tau^2$ to a
variance only when a nonvanishing tied score contrast survives under the Gibbs law.
The full tangent variance can additionally contain the square of the Volterra factor
in \eqref{eq:tangent-inverse-bound}; conversely, Gibbs concentration at a separated,
strongly concave maximizer can cancel part of the local factor.
\end{theorem}

\begin{proof}
For fixed $\tau$, \cref{prop:parabolic-calculus} makes the entropy-cancelled residual
$C^2$.  Its value derivative is $D_V\mathfrak F=I-\mathcal L_\tau$, where
\eqref{eq:parabolic-calculus-causal} bounds every causal power of
$\mathcal L_\tau$ by the corresponding Gamma-denominator term.  Induction gives
\[
 \|\mathcal L_\tau^mW\|_{\mathfrak V,[t,T]}
 \le \frac{[C(1+\tau^{-1})]^m(T-t)^{m/2}}
               {\Gamma(m/2+1)}\|W\|_{\mathfrak V}.
\]
Thus the Neumann series converges in operator norm and is the inverse of
$I-\mathcal L_\tau$; in particular,
\begin{equation}\label{eq:tangent-inverse-bound}
 \|(D_V\mathfrak F)^{-1}\|
 \le C E_{1/2}\!\left(C(1+\tau^{-1})T^{1/2}\right).
\end{equation}
The $C^2$ Banach implicit-function theorem applies and gives
\[
 \|\dot V[u]\|_{\mathfrak V}\le S_\tau\|u\|,
 \qquad
 \|\dot\pi[u]\|_{\mathfrak P}\le C\tau^{-1}(1+S_\tau)\|u\|.
\]
The $C^1$ implicit map gives \eqref{eq:frechet-remainder}; for
$v_n=\widehat\theta_n-\theta_0=O_{\Pp}(n^{-1/2})$, its scaled remainder is
$o_{\Pp}(1)$.  Direct differentiation gives
\eqref{eq:dot-Z}--\eqref{eq:tangent-V2}; the constant term in the KL derivative
vanishes because $\int\dot\pi=0$.  The continuous mapping theorem on the
finite-dimensional range of $D\Psi_\tau(\theta_0)$ then proves
\eqref{eq:equilibrium-functional-clt}, without a separability claim for the ambient
supremum spaces \cite{VanDerVaart1998}.  Classicality follows from
\cref{prop:mild-to-classical}.
\end{proof}

\begin{corollary}[A self-contained open class]
\label{cor:local-existence-class}
Fix $\tau>0$ and let a finite-dimensional family $M_\theta$ satisfy
\cref{ass:eehjb-data}.  At $\theta=0$, suppose
\[
 b_0(t,x,a)=\bar b(t,x),
 \qquad r_0(y,t,x,a)=\bar r(y,t,x).
\]
The two linear terminal problems at $\theta=0$ have a unique classical pair $V_0$,
and $\pi_0=\mu$.  There is $\varepsilon>0$ such that every
$|\theta|<\varepsilon$ has a unique equilibrium in a fixed $\mathfrak V$-neighborhood
of $V_0$.  It is classical, and the value and policy are $C^2$ in $\theta$.
\end{corollary}

\begin{proof}
At the base model the score is constant over actions for every candidate, so
$\pi_{0,V}=\mu$ and entropy cancellation makes
$\mathcal E_{0,\tau}(V)=V_0$ independent of $V$.  Hence
$D_V\mathfrak F(0,V_0)=I$.  The mild-evaluation proposition and the Banach
implicit-function theorem give the local branch and its uniqueness.  Classicality
follows from \cref{prop:mild-to-classical}.
\end{proof}

The class is genuinely state dependent: bounded smooth perturbations of both reward
and controlled drift give a spatially varying first policy derivative.  A concrete
two-parameter construction, including direct verification of
\cref{ass:eehjb-data}, is given in the supplement.

The theorem keeps $\tau>0$ fixed and does not justify a triangular array
$\tau_n\downarrow0$ without uniform control of the local neighborhood and quadratic
remainder.  Precise sufficient conditions are stated in the supplement.  Likewise, a
numerical equilibrium error $o_{\Pp}(n^{-1/2})$ preserves the limit law, but turning
this observation into a stopping rule requires a separate a posteriori error bound.

\section{Exact models and feedback graphs}\label{sec:exact}

Section~\ref{sec:volterra} leaves the upper bound unmatched unless a positive return mode
is present.  The two models below make that return mechanism explicit.  The affine model
gives a closed nonlinear statistical transition, but its linearly growing payoffs place
it outside the bounded framework of \cref{sec:stability}.  The trigonometric model has
bounded smooth data and a uniformly nondegenerate diffusion.

\subsection{An affine reference-state diffusion}

Let $\A=\{-1,+1\}$, let $\mu$ be uniform, and write
\begin{equation}\label{eq:binary-mean}
 m^\pi(t,x)=\int_\A a\,\pi(t,x,da).
\end{equation}
The two-point space is compact and $\mu$ charges both points, so this is a direct
instance of the action-space setup in \cref{sec:model}; no limiting argument is used.
For $\sigma>0$, consider
\begin{equation}\label{eq:affine-state}
 dX_s=m^\pi(s,X_s)\,ds+\sigma\,dW_s.
\end{equation}
The self with frozen reference state $y$ evaluates
\begin{align}
 J^{\pi,\xi}(y;t,x)=\E_{t,x}\bigg[\int_t^T
 \bigg\{\beta(X_s-y)m^\pi(s,X_s)
 -\tau\KL\bigl(\pi(s,X_s)\Vert\mu\bigr)\bigg\}\,ds
 +\xi X_T\bigg],                                      \label{eq:affine-objective}
\end{align}
where $\beta>0$.  Immediate rewards are tied on the acting diagonal $y=x$.
Nevertheless, the current drift changes the state from which future selves evaluate
their reference-state reward.

\begin{theorem}[Exact affine cascade]\label{thm:exact-affine}
For every $(\tau,\xi)\in(0,\infty)\times\R$, the model
\eqref{eq:affine-state}--\eqref{eq:affine-objective} has a unique equilibrium in the
affine class
\begin{equation}\label{eq:affine-value-ansatz}
 V^\xi(y;t,x)=w_\xi(t)(x-y)+\xi y+c_\xi(t).
\end{equation}
It is state independent and satisfies
\begin{align}
 m_\tau^\xi(t)&=\tanh\!\left(\frac{w_\xi(t)}\tau\right),
 &w_\xi'(t)+\beta m_\tau^\xi(t)&=0,
 &w_\xi(T)&=\xi.                                      \label{eq:affine-ode}
\end{align}
More explicitly,
\begin{align}
 \sinh\!\left(\frac{w_\xi(t)}\tau\right)
 &=e^{\beta(T-t)/\tau}\sinh(\xi/\tau),               \label{eq:exact-w}\\
 m_\tau^\xi(t)
 &=\frac{z_\tau^\xi(t)}{\sqrt{1+z_\tau^\xi(t)^2}},
 &z_\tau^\xi(t)&=e^{\beta(T-t)/\tau}\sinh(\xi/\tau).
                                                               \label{eq:exact-m}
\end{align}
At the tied model $\xi=0$, the policy susceptibility is
\begin{equation}\label{eq:exact-affine-susceptibility}
 \chi_\tau(t):=\left.\partial_\xi m_\tau^\xi(t)\right|_{\xi=0}
 =\frac1\tau e^{\beta(T-t)/\tau}.
\end{equation}
\end{theorem}

\begin{proof}
Under a state-independent policy with mean $m(t)$, substituting
\eqref{eq:affine-value-ansatz} into the linear evaluation PDE and comparing the
coefficient of $x-y$ gives
$w'(t)+\beta m(t)=0$ and $w(T)=\xi$.  The flow-state partial derivative is
$V_x=w(t)$.  By contrast, the total derivative of the diagonal map is
$dV(x;t,x)/dx=V_x+V_y=\xi$; replacing the former by the latter would incorrectly
erase the intertemporal effect.

At the acting diagonal, action $a$ has score $aw(t)$, so the binary Gibbs identity
gives $m(t)=\tanh(w(t)/\tau)$ and \eqref{eq:affine-ode}.  The remaining scalar
coefficient solves a linear terminal equation and does not affect the policy.
Global Lipschitz continuity gives uniqueness; the resulting bounded deterministic
drift makes $X$ Gaussian with finite moments, while
$\KL(\pi_s\Vert\mu)\le\log2$, proving admissibility.  Finally,
\[
 \frac{d}{dt}\log\left|\sinh\frac{w_\xi(t)}\tau\right|
 =-\frac\beta\tau
\]
away from zero; continuity covers the zero solution.  Integration, followed by
$\tanh u=\sinh u/\sqrt{1+\sinh^2u}$, proves
\eqref{eq:exact-w}--\eqref{eq:exact-m}.  Differentiation at $\xi=0$ gives
\eqref{eq:exact-affine-susceptibility}.
\end{proof}

The formula yields an exact joint statistical and annealing limit, without requiring a
uniform delta-method remainder.

\begin{theorem}[Root-$n$ phase transition]\label{thm:root-n-phase}
Let $\widehat\xi_n$ satisfy
$Z_n:=\sqrt n\,\widehat\xi_n\Rightarrow Z$, where $Z$ is nondegenerate and
$\Pp(Z=0)=0$, and let
$\tau_n\downarrow0$ with $\sqrt n\,\tau_n\to\infty$.  Define
\begin{equation}\label{eq:a-n}
 a_n=\frac{e^{\beta T/\tau_n}}{\sqrt n\,\tau_n}.
\end{equation}
Then the time-zero equilibrium obeys
\begin{equation}\label{eq:three-phase-law}
 m_{\tau_n}^{\widehat\xi_n}(0)\Rightarrow
 \begin{cases}
 0, & a_n\to0,\\[1mm]
 \displaystyle\frac{aZ}{\sqrt{1+a^2Z^2}},
     &a_n\to a\in(0,\infty),\\[3mm]
 \operatorname{sign}(Z),&a_n\to\infty.
 \end{cases}
\end{equation}
The exact unit-amplification temperature and its first-order expansion are
\begin{equation}\label{eq:lambert-critical-temperature}
 \tau_n^\star=\frac{\beta T}{W(\beta T\sqrt n)}
 \sim\frac{2\beta T}{\log n},
\end{equation}
where $W$ is the same principal real branch as above.  Equivalently, a
nonunit critical limit
$a\in(0,\infty)$ is characterized by
\begin{equation}\label{eq:exact-critical-centering}
 \frac{\beta T}{\tau_n}-\frac12\log n-\log\tau_n\longrightarrow\log a.
\end{equation}
At any such critical scale and every fixed $t\in(0,T]$,
$m_{\tau_n}^{\widehat\xi_n}(t)\to0$ in probability, although the time-zero policy has
the nondegenerate limit in the middle line of \eqref{eq:three-phase-law}.
\end{theorem}

\begin{proof}
Since $Z_n=O_{\Pp}(1)$ and $\sqrt n\tau_n\to\infty$,
$\widehat\xi_n/\tau_n\to0$ in probability.  Consequently,
\[
 \sinh\left(\frac{\widehat\xi_n}{\tau_n}\right)
 =\frac{Z_n}{\sqrt n\tau_n}\{1+o_{\Pp}(1)\}.
\]
Thus \eqref{eq:exact-m} gives
$z_{\tau_n}^{\widehat\xi_n}(0)=a_nZ_n\{1+o_{\Pp}(1)\}$.
The three conclusions follow from the continuous mapping theorem; in the last case use
$\Pp(Z=0)=0$.  Solving $a_n=1$ gives
$x_ne^{x_n}=\beta T\sqrt n$ with $x_n=\beta T/\tau_n$, proving
\eqref{eq:lambert-critical-temperature}.  Taking logarithms gives
\eqref{eq:exact-critical-centering}.  Finally, at a fixed $t>0$ the corresponding
amplification is $a_ne^{-\beta t/\tau_n}\to0$.
\end{proof}

The leading and two-term logarithmic approximations used below are
\begin{equation}\label{eq:critical-temperature-approximations}
 \tau_{n,1}=\frac{2\beta T}{\log n},\qquad
 L_n=\log(\beta T\sqrt n),\qquad
 \tau_{n,2}=\frac{\beta T}{L_n-\log L_n}.
\end{equation}
The leading approximation misses the critical centering.  Since
$W(x)=\log x-\log\log x+o(1)$, the log--log term determines whether the amplified noise
has a finite nonzero limit.  The transition is localized near the initial time: at the
critical scale, the initial policy remains sample dependent while each fixed later-time
policy converges to the reference mixture.

Although $1/\log n$ also occurs in simulated annealing \cite{Hajek1988}, there it is
an algorithmic-time cooling law tied to energy barriers.  Here it balances root-$n$
statistical noise against a causal equilibrium condition number and changes the policy's
statistical limit; no mixing-time claim is involved.

\subsection{Coupled bounded modes and feedback cycles}
\label{subsec:coupled-modes}

The scalar example extends to coupled modes without losing its closed linear response.
Let $d\ge1$, $\A=[-1,1]^d$, and $\mu=\mu_0^{\otimes d}$, where $\mu_0$ is a symmetric
nondegenerate probability measure with full support on $[-1,1]$.  Put
\begin{equation}\label{eq:ell-v}
 v=\int_{-1}^1 a^2\mu_0(da)>0,\qquad
 \ell(z)=\frac{\int_{-1}^1 a e^{za}\mu_0(da)}
               {\int_{-1}^1 e^{za}\mu_0(da)}.
\end{equation}
For the uniform law, $v=1/3$ and $\ell(z)=\coth z-z^{-1}$, with the value at zero
defined by continuity.  Thus this model has a genuinely continuous action space.

For $K\in\R^{d\times d}$, consider
\begin{align}
 dX_s&=m^\pi(s,X_s)\,ds+\sigma\,dW_s,
 &m^\pi(t,x)&=\int_\A a\,\pi(t,x,da),                 \label{eq:coupled-state}\\
 r(y,s,x,a)&=\sin(x-y)^\top Ka,
 &F_{\boldsymbol\varepsilon}(y,x)
   &=\boldsymbol\varepsilon^\top\sin(x-y),             \label{eq:coupled-data}
\end{align}
where sine acts componentwise, $\sigma>0$, and $W$ is $d$-dimensional.  The data are bounded and
smooth in the state variables, and the diffusion is uniformly elliptic.  In the
notation of \cref{sec:stability}, we take $\delta\equiv1$, $h\equiv0$, and
$G\equiv0$.

\begin{theorem}[Coupled bounded modes]\label{thm:coupled-modes}
Write $\nu=\sigma^2/2$.  The model
\eqref{eq:coupled-state}--\eqref{eq:coupled-data} has a unique continuous
equilibrium within the state-independent class.  It satisfies
\begin{align}
 \pi_t^{\boldsymbol\varepsilon}(da)
 &=\bigotimes_{i=1}^d
 \frac{e^{a_iZ_i(t)/\tau}\mu_0(da_i)}
      {\int_{-1}^1e^{uZ_i(t)/\tau}\mu_0(du)},\notag\\
 m_i(t)&=\ell\!\left(\frac{Z_i(t)}\tau\right),        \label{eq:coupled-gibbs}\\
 Z_i(t)&=\varepsilon_i e^{-\nu(T-t)}\cos M_i(t,T)
 +\int_t^T e^{-\nu(s-t)}\cos M_i(t,s)(Km(s))_i\,ds,\notag\\
 M_i(t,s)&=\int_t^s m_i(u)\,du.                       \label{eq:coupled-volterra}
\end{align}
At $\boldsymbol\varepsilon=0$, the susceptibility matrix is
\begin{equation}\label{eq:matrix-susceptibility}
 \chi_\tau(t):=
 D_{\boldsymbol\varepsilon}m_{\boldsymbol\varepsilon}(t)\big|_0
 =\frac v\tau
 \exp\left\{\left(-\nu I+\frac v\tau K\right)(T-t)\right\}.
\end{equation}
If $K$ is symmetric, then
\begin{equation}\label{eq:symmetric-matrix-norm}
 \norm{\chi_\tau(t)}_2
 =\frac v\tau
 \exp\left\{\left(\frac{v\lambda_{\max}(K)}\tau-\nu\right)(T-t)\right\}.
\end{equation}
If $K$ and $\boldsymbol\varepsilon$ range over compact sets and
$\mathcal I\subset(0,\bar\tau]$, these equilibria satisfy
\cref{ass:eehjb-data,ass:eehjb-envelope} with comparison constants independent of
$\tau\in\mathcal I$.  Parameter derivatives are asserted only at fixed temperature.
\end{theorem}

\begin{proof}
For deterministic $m$, Gaussian convolution gives, coordinatewise,
\[
 \E_{t,x}\sin(X_s^i-y_i)
 =e^{-\nu(s-t)}\sin\{x_i-y_i+M_i(t,s)\}.
\]
Differentiating policy evaluation in the flow state and setting $y=x$ gives
\eqref{eq:coupled-volterra}.  The acting score
is $a^\top Z(t)$.  Since the reference measure is a product, the Gibbs law factorizes
and its mean is \eqref{eq:coupled-gibbs}.

A complex score turns the Volterra equation into a finite-dimensional ODE.  Define
\[
 \zeta_i(t)=\varepsilon_i e^{-\nu(T-t)+iM_i(t,T)}
 +\int_t^T e^{-\nu(s-t)+iM_i(t,s)}(Km(s))_i\,ds.
\]
Then $Z_i=\operatorname{Re}\zeta_i$.  In reverse time, writing
$\zeta_i(T-u)=x_i(u)+iy_i(u)$ turns the Volterra equation into
\begin{equation}\label{eq:coupled-score-ode}
\begin{aligned}
 \dot x=-\nu x-m\odot y+Km,
 &\qquad \dot y=-\nu y+m\odot x,\\
 m_i=\ell(x_i/\tau),
 &\qquad (x(0),y(0))=(\boldsymbol\varepsilon,0).
\end{aligned}
\end{equation}
The vector field is locally Lipschitz and has linear growth because $|m_i|\le1$.
Thus \eqref{eq:coupled-score-ode}, and equivalently the Volterra system, has one global
solution.  This proves existence and uniqueness in the state-independent class.

Symmetry of $\mu_0$ gives $\ell(0)=0$ and $\ell'(0)=v$.  Differentiation at the tied
model yields
\[
 \chi_\tau(t)=\frac v\tau e^{-\nu(T-t)}I
 +\frac v\tau\int_t^T e^{-\nu(s-t)}K\chi_\tau(s)\,ds.
\]
In reverse time this is the mild equation with generator
$-\nu I+(v/\tau)K$, proving \eqref{eq:matrix-susceptibility}.  Diagonalizing a
symmetric $K$ gives \eqref{eq:symmetric-matrix-norm}.

For the final assertion, the uncontrolled evolution is the heat semigroup, the
trigonometric primitives are bounded and smooth, and the induced drift is spatially
constant with $|m_i|\le1$.  Gaussian convolution gives an explicit bounded value,
while $\tau\KL(\pi_t\Vert\mu)\le\osc_a(a^\top Z(t))$ and
$|Z(t)|\le C(|\boldsymbol\varepsilon|+T\|K\|)$.  These bounds verify
\cref{ass:eehjb-data,ass:eehjb-envelope}; the complete value formula and derivative
check are recorded in the supplement.
\end{proof}

\begin{corollary}[Acyclic and cyclic feedback]\label{cor:feedback-graph}
Assume $K\ge0$ and fix $t<T$.  Give $K$ the directed graph with an edge $j\to i$ when
$K_{ij}>0$, and put $h=T-t$.
\begin{enumerate}[label=\textup{(\roman*)},leftmargin=2em]
\item If the graph is acyclic with longest directed path of length $L$, then
\begin{equation}\label{eq:dag-susceptibility}
 \chi_\tau(t)=\frac v\tau e^{-\nu h}
 \sum_{r=0}^L\frac{(vh)^r}{r!\tau^r}K^r,
 \qquad \norm{\chi_\tau(t)}_2=\Theta(\tau^{-(L+1)}).
\end{equation}
If $K^Lc\ne0$, the response to $n^{-1/2}c$ has boundary
\begin{equation}\label{eq:dag-statistical-scale}
 \tau_n\asymp n^{-1/\{2(L+1)\}}.
\end{equation}
\item If the graph contains a directed cycle, then $\rho(K)>0$ and
\begin{equation}\label{eq:cycle-susceptibility}
 \rho(\chi_\tau(t))
 =\frac v\tau\exp\left\{-\nu h+\frac{v\rho(K)h}\tau\right\}.
\end{equation}
Every induced matrix norm is bounded below by the right-hand side.  For a Perron
vector, the response to $n^{-1/2}r$ has unit amplification at
\begin{equation}\label{eq:cycle-critical-temperature}
 \tau_{n,\rm cyc}^\star
 =\frac{v\rho(K)h}{W\!\left(\rho(K)h e^{\nu h}\sqrt n\right)}
 \sim\frac{2v\rho(K)h}{\log n}.
\end{equation}
\end{enumerate}
For the full policy derivative,
\begin{equation}\label{eq:policy-matrix-lower}
 \left\|D_{\boldsymbol\varepsilon}\pi_t\big|_0
 \right\|_{\ell^\infty\to{\rm TV}}
 \ge \|\chi_\tau(t)\|_{\infty\to\infty}
 \ge \rho(\chi_\tau(t)).
\end{equation}
The matching upper bound differs by at most $d/v$; hence the policy derivative has
the same polynomial order in the acyclic case and exponential rate in the cyclic
case.  The norm comparison is proved in the supplement.
\end{corollary}

\begin{proof}
An acyclic nonnegative matrix is nilpotent after a simultaneous permutation of rows
and columns.  Path expansion gives $K^{L+1}=0$ and $K^L\ne0$, so
\eqref{eq:dag-susceptibility} follows by expanding the exponential.  A positive
directed cycle implies $\rho(K)>0$.  Spectral mapping applied to
\eqref{eq:matrix-susceptibility} gives \eqref{eq:cycle-susceptibility}.  Scaling the
leading DAG term by $n^{-1/2}$ proves \eqref{eq:dag-statistical-scale}.  Along a
Perron vector, equating the multiplier in \eqref{eq:cycle-susceptibility} to
$\sqrt n$ gives $xe^x=\rho(K)he^{\nu h}\sqrt n$ with
$x=v\rho(K)h/\tau$, which proves \eqref{eq:cycle-critical-temperature}.
\end{proof}

The graph also controls the nonlinear response.

\begin{corollary}[Nonlinear acyclic decay and cooperative selection]
\label{cor:nonlinear-graph}
Assume $K\ge0$.
\begin{enumerate}[label=\textup{(\roman*)},leftmargin=2em]
\item If the graph is acyclic with longest path $L$, then
\begin{equation}\label{eq:dag-nonlinear-bound}
 |m_{\boldsymbol\varepsilon}(t)|
 \le\frac1\tau\sum_{r=0}^L
 \frac{(T-t)^r}{r!\tau^r}K^r|\boldsymbol\varepsilon|.
\end{equation}
In particular, $\|\boldsymbol\varepsilon_\tau\|=O(e^{-\gamma/\tau})$ with
$\gamma>0$ implies $m_{\boldsymbol\varepsilon_\tau}(t)\to0$.
\item Suppose
$\kappa_*:=\min_i(K\mathbf1)_i>0$,
$\sup\operatorname{supp}\mu_0=1$, and $T<\pi/2$.  Set
$a_\sigma=e^{-\nu T}\cos T$.  For
$\boldsymbol\varepsilon_\tau=\tau e^{-\gamma/\tau}\mathbf1$ and fixed $t<T$, if
\begin{equation}\label{eq:collective-selection-condition}
 0<\gamma<a_\sigma v\kappa_*(T-t),
\end{equation}
then $m_{\boldsymbol\varepsilon_\tau,i}(t)\to1$ for every $i$.
\end{enumerate}
\end{corollary}

\begin{proof}
Since $\ell'(z)$ is the variance under an exponential tilt of $\mu_0$,
$|\ell'(z)|\le1$.  Equations
\eqref{eq:coupled-gibbs}--\eqref{eq:coupled-volterra} give
\[
 |m(t)|\le\tau^{-1}|\boldsymbol\varepsilon|
 +\tau^{-1}\int_t^T K|m(s)|\,ds.
\]
Positive iteration stops after $K^L$ and proves \eqref{eq:dag-nonlinear-bound}.

For the second claim, the positive orthant is invariant in
\eqref{eq:coupled-score-ode}.
On a face $x_i=0$, one has $m_i=0$ and $\dot x_i=(Km)_i\ge0$.  Thus $x,m\ge0$.
On a face $y_i=0$, one has $\dot y_i=m_ix_i\ge0$, so $y\ge0$ as well.
Since $0\le M_i(t,s)\le s-t<T<\pi/2$, the score satisfies, componentwise,
\begin{equation}\label{eq:coupled-score-lower}
 Z_i(t)\ge a_\sigma\varepsilon_i
 +a_\sigma\int_t^T(Km(s))_i\,ds.
\end{equation}
Let $\mathcal P$ be the positive Volterra map obtained by applying
$\ell(\cdot/\tau)$ to the right-hand side of
\eqref{eq:coupled-score-lower}.  The exact solution satisfies
$m\ge\mathcal Pm$.  Since $m\ge0$ and $\mathcal P$ is positive and monotone,
$0\le\mathcal P0\le\mathcal Pm\le m$; induction gives
$m\ge\mathcal P^k0$ for every $k$.
The row-sum bound $K\mathbf1\ge\kappa_*\mathbf1$ shows inductively that these
vector iterates dominate the scalar Picard iterates whose limit is
$\underline m$.  Hence $m_i(t)\ge\underline m(t)$, where
\[
 \underline w(t)=a_\sigma\tau e^{-\gamma/\tau}
 +a_\sigma\kappa_*\int_t^T\ell(\underline w(s)/\tau)\,ds,
 \qquad \underline m(t)=\ell(\underline w(t)/\tau).
\]
For $q_\tau(u)=\underline w(T-u)/\tau$,
\[
 q_\tau'(u)=\frac{a_\sigma\kappa_*}{\tau}\ell(q_\tau(u)),
 \qquad q_\tau(0)=a_\sigma e^{-\gamma/\tau}.
\]
Since $\ell(q)/q\to v$ at zero, the time needed to reach any fixed small
$\delta>0$ is
\[
 \frac{\tau}{a_\sigma\kappa_*}
 \int_{a_\sigma e^{-\gamma/\tau}}^\delta\frac{dq}{\ell(q)}
 =\frac{\gamma}{a_\sigma v\kappa_*}+o(1).
\]
Condition \eqref{eq:collective-selection-condition} leaves positive time after this
hitting point.  On that interval $\ell(q)\ge\ell(\delta)>0$, so
$q_\tau(T-t)\to\infty$.  The support assumption gives $\ell(q)\to1$.
\end{proof}

The restriction $T<\pi/2$ is used only in part~\textup{(ii)}: it keeps the cosine
kernel positive.  On a longer horizon the return can change sign, so the monotone
comparison above gives no selection conclusion.

If $K\ge0$ is irreducible, let $r\gg0$ be its Perron vector and equip $\R^d$ with the
weighted norm $\|x\|_r=\max_i|x_i|/r_i$.  Then $\|K\|_r=\rho(K)$, and the tangent
operator in \eqref{eq:matrix-susceptibility} satisfies
\eqref{eq:mode-subeigen} with $\alpha=1$, $\psi(t)=e^{\nu t}r$, and
$\kappa_-=\kappa=v\rho(K)$.  Thus the abstract upper and lower exponential rates are
both attained by a bounded uniformly elliptic model in every state dimension.  A
reducible matrix gives the same conclusion after restriction to an irreducible Perron
class.

\section{Time discretization}\label{sec:discretization}

A finite time grid has a strictly triangular influence matrix and hence a polynomial
condition number at fixed $N$.  The continuous-time operator can have infinitely many
nonzero causal powers and exponential conditioning, so refinement need not commute with
$\tau\downarrow0$.  We quantify this mismatch by applying a uniform right-endpoint rule
to the coupled tangent equation after factoring out the Brownian damping.

\begin{proposition}[Graph-dependent mesh stiffness]
\label{prop:graph-mesh}
For the coupled model with $K\ge0$, set
$Y_i=e^{\nu(T-t_i)}\chi_{\tau,N}(t_i)$ and use the right-endpoint rule
\[
 Y_i=\frac v\tau I+\frac{v\Delta}{\tau}
       \sum_{j=i}^{N-1}K Y_{j+1},
 \qquad Y_N=\frac v\tau I.
\]
The resulting $N$-step time-zero susceptibility is
\begin{equation}\label{eq:matrix-discrete-susceptibility}
 \chi_{\tau,N}(0)=\frac v\tau e^{-\nu T}
\left(I+\frac{vT}{N\tau}K\right)^N.
\end{equation}
For any fixed induced matrix norm, write
\begin{equation}\label{eq:matrix-relative-mesh-error}
 \mathfrak e_{\tau,N}
 :=\frac{\|\chi_{\tau,N}(0)-\chi_\tau(0)\|}
          {\|\chi_\tau(0)\|}.
\end{equation}
If the graph is acyclic with longest path $L$, then
\begin{equation}\label{eq:dag-discrete-susceptibility}
 \chi_{\tau,N}(0)=\frac v\tau e^{-\nu T}
 \sum_{r=0}^L\binom Nr\left(\frac{vT}{N\tau}\right)^rK^r.
\end{equation}
For fixed $N\ge L$, the ratio of the leading low-temperature coefficient to its
continuous counterpart is
\begin{equation}\label{eq:dag-leading-grid-ratio}
 q_{N,L}=\frac{L!\binom NL}{N^L}
 =\prod_{j=0}^{L-1}\left(1-\frac jN\right).
\end{equation}
When $L\le1$, the discrete and continuous susceptibilities agree for every
$N\ge1$.  When $L\ge2$ and $N\ge L$, one has
$\mathfrak e_{\tau,N}\to0$ jointly with $\tau\downarrow0$ if and only if
$N\to\infty$; no coupling between $N$ and $\tau$ is needed.

If the graph contains a cycle and $r$ is a Perron vector, then along $r$ the ratio of
the discrete to continuous response is
\begin{equation}\label{eq:cycle-grid-ratio}
 \frac{\|\chi_{\tau,N}(0)r\|}{\|\chi_\tau(0)r\|}
 =\left(1+\frac{x}{N}\right)^Ne^{-x},
 \qquad x=\frac{v\rho(K)T}{\tau}.
\end{equation}
Along any sequence $\tau\downarrow0$ and $N\to\infty$, this ratio tends to one if and
only if $N\tau^2\to\infty$.
\end{proposition}

\begin{proof}
Backward substitution of the linearized grid equation gives
\eqref{eq:matrix-discrete-susceptibility}.  If $K^{L+1}=0$, the binomial series stops
at $L$, giving \eqref{eq:dag-leading-grid-ratio}; coefficientwise comparison of the
two finite sums proves the stated DAG equivalence.  For a Perron vector,
\eqref{eq:matrix-discrete-susceptibility} reduces to \eqref{eq:cycle-grid-ratio}, and
$u-u^2/2\le\log(1+u)\le u-u^2/[2(1+u)]$ gives
\[
 -\frac{x^2}{2N}\le
 \log\left\{\left(1+\frac{x}{N}\right)^Ne^{-x}\right\}
 \le-\frac{x^2}{2N(1+x/N)}.
\]
The bounds are equivalent to $x^2/N\to0$, hence to $N\tau^2\to\infty$; the short
necessity and subsequence arguments for both cases are recorded in the supplement.
\end{proof}

\subsection{Numerical illustrations}\label{subsec:numerics}

All numerical calculations are reproducible.  The Python file
\texttt{sicon\_volterra\_experiments.py} generates every reported number and both figures.
The Monte Carlo calculation uses $100{,}000$
common Gaussian draws with seed $20260718$; all remaining checks are deterministic.
No constant in a theoretical curve is fitted to the output.

For \cref{fig:feedback-graph}, $d=6$, $T=1$, $\sigma=0.8$, and $\mu_0$ is uniform on
$[-1,1]$, so $v=1/3$.  The chain matrix has $K_{i+1,i}=1$ for $i=1,\ldots,5$ and all
other entries zero.  The cyclic matrix adds $K_{1,6}=1$.

\begin{figure}[t]
\centering
\includegraphics[width=.98\textwidth]{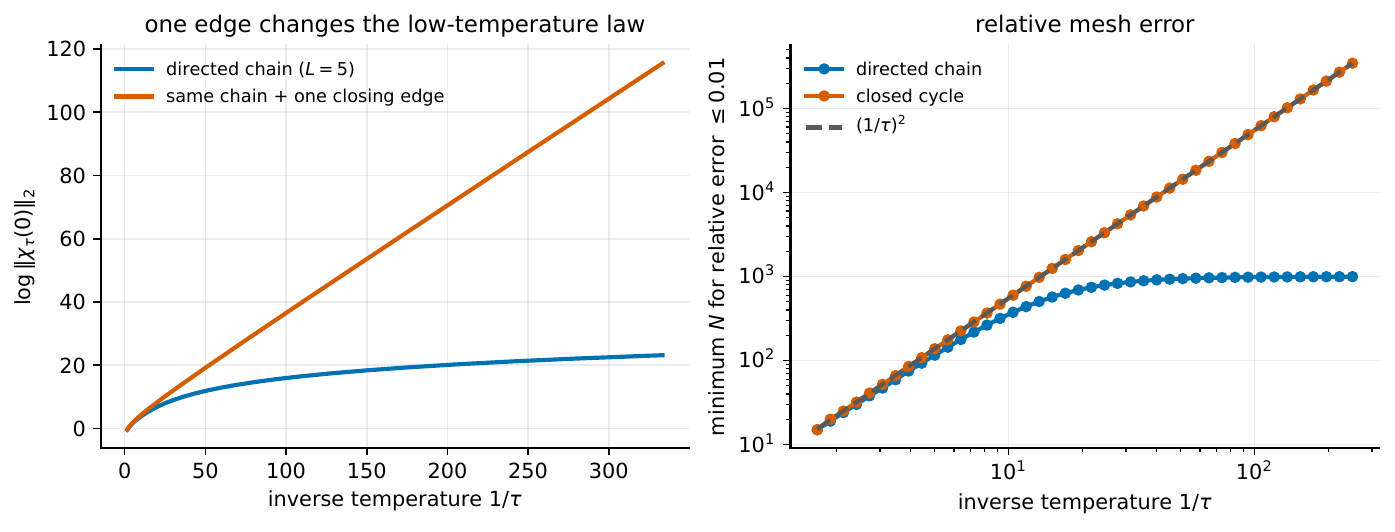}
\caption{Susceptibility and grid cost for a six-node chain and the cycle formed by adding
one closing edge.  Left: $\log\|\chi_\tau(0)\|_2$ is asymptotically linear in
$\log(1/\tau)$ for the chain and in $1/\tau$ for the cycle.  Right: the smallest $N$ for
which the relative response error is at most $0.01$.  The chain uses the full matrix
$2$-norm error in \eqref{eq:matrix-relative-mesh-error}, while the cycle uses the
Perron-mode response error in \eqref{eq:cycle-grid-ratio}.  The chain cost approaches
a temperature-independent level, whereas the cyclic Perron cost scales as
$\tau^{-2}$.}
\label{fig:feedback-graph}
\end{figure}

For \cref{fig:statistical-phase}, $\beta=T=1$, $Z\sim N(0,1)$, and
$\widehat\xi_n=Z/\sqrt n$.  The same Gaussian draws are used across the three
temperature schedules at each $n$.

At the largest simulated size $n=2^{26}$, the mean absolute time-zero magnetizations
are $0.0503$, $0.5229$, and $0.9738$ in the stable, critical, and selection schedules.
At $t=T/2$ they are $0.0048$, $0.0233$, and $0.1465$, consistent with the fixed-time
convergence in \cref{thm:root-n-phase}.  At $n=10^{14}$, the relative errors of the
leading and two-term logarithmic approximations in
\eqref{eq:critical-temperature-approximations} are $16.2\%$ and $1.32\%$.

\begin{figure}[t]
\centering
\includegraphics[width=.84\textwidth]{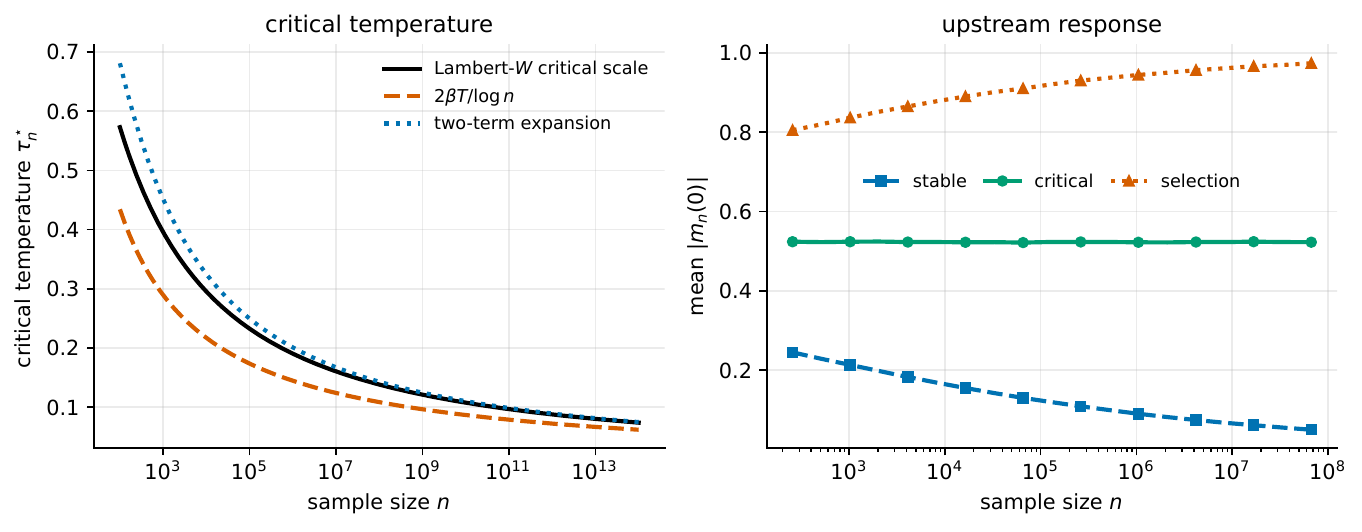}
\caption{Critical-temperature transition.  Left: the exact Lambert-$W$ scale and the
approximations in \eqref{eq:critical-temperature-approximations}.  Right: Monte Carlo
estimates of $\E|m_{\tau_n}^{\widehat\xi_n}(0)|$ for
$\tau_n/\tau_n^\star\in\{1.5,1,0.7\}$.}
\label{fig:statistical-phase}
\end{figure}

\section{Discussion}\label{sec:conclusion}

On a DAG, a perturbation passes through only finitely many future selves.  Under aligned
forcing, a positive cycle permits repeated returns and can change polynomial
susceptibility into exponential susceptibility.  The bounded diffusion realizes both
rates in arbitrary finite dimension.  It also shows why a grid that is accurate at fixed
temperature may fail under cooling: along a cyclic Perron mode, relative consistency requires
$N\tau^2\to\infty$.

Two limits of the present theory remain.  The general parabolic argument gives
$e^{C/\tau^2}$, whereas the explicit graph model attains $e^{C/\tau}$; whether state
dependence can attain the larger rate is open.  The statistical result assumes
finite-dimensional model estimation and fixed positive temperature, except in the
exactly solvable affine model.  Action-dependent diffusion, nonsmooth estimators, and
graph criteria for state-dependent equilibria remain unresolved.

\section*{Declaration of generative-AI assistance}
OpenAI ChatGPT/Codex (July 2026) assisted with literature searches, drafting, editing,
mathematical checks, and numerical code.  The author assumes responsibility for all
content.

\section*{Code availability}
The accompanying reproducibility archive contains the seeded Python code and numerical
output used to produce both figures and every reported numerical value.

\noindent\textbf{Supplementary material.}
The supplement contains the detailed technical closures and classical bootstrap used
in Sections~4--6.

\bibliographystyle{siamplain}
\bibliography{references}

\end{document}


\maketitle

This supplement records the entropy cancellation, precise strong norms,
bounded-drift and source-difference closures, second-order mild remainder,
fixed-temperature scope conditions, bounded-model envelope, mesh-stiffness closure,
and classical bootstrap used in the main paper.
We use the notation of Section~4 of the main paper, in particular
$\mathfrak V$, $\mathfrak V_\eta$, $\mathcal E_{\theta,\tau}$,
$P^0_{t,s}$, and $\mathsf T$.  All constants in this supplement are
fixed-temperature constants unless stated otherwise.

\section{Positive-axis Mittag--Leffler growth}

For every $\alpha>0$,
\begin{equation}\label{eq:supp-mittag-log-growth}
 \log E_\alpha(z)\sim z^{1/\alpha},\qquad z\to\infty,\quad z>0.
\end{equation}
Write $z=y^\alpha$, set
$a_m(y)=y^{\alpha m}/\Gamma(\alpha m+1)$, and put
$r=\alpha m/y$.  If $m_y=\lfloor y/\alpha\rfloor$, Stirling's formula gives
\[
 y^{-1}\log a_{m_y}(y)\longrightarrow1.
\]
Since all terms are positive, this proves the required lower bound for
$y^{-1}\log E_\alpha(y^\alpha)$.

For the upper bound, the two-sided Stirling inequalities imply, uniformly when
$\alpha m\ge1$,
\[
 \log a_m(y)\le y\phi(r)+C\log(1+\alpha m),
 \qquad \phi(r):=r(1-\log r)\le1.
\]
Fix $R>1$ for the moment.  There are $O_R(y)$ indices with
$\alpha m\le Ry$, so their sum is at most
$\exp\{y+O_R(\log y)\}$.  To control the remaining infinite tail, use the
fixed-$\alpha$ gamma-ratio bound
\[
 \frac{a_{m+1}(y)}{a_m(y)}
 =y^\alpha\frac{\Gamma(\alpha m+1)}
                    {\Gamma(\alpha m+\alpha+1)}
 \le C_\alpha\left(\frac{y}{\alpha m}
              \right)^\alpha .
\]
Choose $R$ so large that the last expression is at most $1/2$ whenever
$\alpha m\ge Ry$.  The tail is then at most twice its first term, hence also
$\exp\{y+O_R(\log y)\}$.  Therefore
\[
 \limsup_{y\to\infty}y^{-1}\log E_\alpha(y^\alpha)\le1.
\]
Together with the lower bound this proves \eqref{eq:supp-mittag-log-growth}.

\section{Short-time spike estimate}

Under the hypotheses of Proposition~2.1 in the main paper, continuity uniformly in the
action makes the spike coefficients bounded on a neighborhood of $(t,x)$.  Let
$\theta_R$ be the first exit of $X^{\epsilon,\rho}$ from the ball of radius $R$ about
$x$.  The Burkholder--Davis--Gundy inequality and the stopped SDE give
\[
 \mathbb E\sup_{t\le s\le(t+\epsilon)\wedge\theta_R}
 |X_s^{\epsilon,\rho}-x|^2\le C_R(\epsilon+\epsilon^2).
\]
For any fixed $\eta<R$, continuity of the path and Chebyshev's inequality therefore
show that the event for the unstopped path equals the corresponding event for the
stopped path, and hence
$\mathbb P(\sup_{t\le s\le t+\epsilon}|X_s^{\epsilon,\rho}-x|>\eta)\to0$.
This is the short-time concentration used before applying right-continuity and Vitali's
theorem to the uniformly integrable spike integrand.

\section{Entropy cancellation}

Define the log-partition Hamiltonian
\begin{equation}\label{eq:supp-log-partition}
 H_\tau(q)=\tau\log\int_\A e^{q(a)/\tau}\mu(da).
\end{equation}
If $\pi=\cG_\tau(q)$, then
\begin{equation}\label{eq:supp-gibbs-value}
 H_\tau(q)=\int q\,d\pi-\tau\KL(\pi\Vert\mu),
 \qquad DH_\tau(q)[h]=\int h\,d\pi.
\end{equation}
In a discounted policy-evaluation equation, take
$q(a)=b(a)\cdot Z+r(x,t,x,a)$ and write
$R_y=\int r(y,t,x,a)\pi(da)$ and $R_x=\int r(x,t,x,a)\pi(da)$.
Here $D_xV$ is the flow-state derivative of
$V^1(\vartheta,t,y,x)$ with $(\vartheta,y)$ frozen.  Then
\begin{align}
 &B\cdot D_xV+\delta(t-\vartheta)
 \{R_y-\tau\KL(\pi\Vert\mu)\}\notag\\
 &\qquad=B\cdot\{D_xV-\delta(t-\vartheta)Z\}
 +\delta(t-\vartheta)\{H_\tau(q)+R_y-R_x\}.          \label{eq:supp-entropy-cancellation}
\end{align}
The Hamiltonian is Lipschitz in $q$ with a temperature-independent constant.  The
off-diagonal difference $R_y-R_x$ has a bounded integrand under the data assumption
of the main paper.  This proves the source estimate used in its stability proof.
Differentiating
\eqref{eq:supp-gibbs-value} also shows why no derivative of the normalizing constant
is missing from the forced tangent equation.

For completeness, here is the cancelled mild evaluation map used in the main paper.
For a candidate $V$, let $\pi_{\theta,V}=\cG_\tau(q_{\theta,V})$ and set
\begin{align*}
 B_{\theta,V}&=\int_\A b_\theta(a)\pi_{\theta,V}(da),\\
 R_{\theta,y,V}&=\int_\A r_\theta(y,\cdot,a)\pi_{\theta,V}(da),\\
 \Phi^1_{\theta,V}
 &=\delta(t-\vartheta)\{H_\tau(q_{\theta,V})
   +R_{\theta,y,V}-R_{\theta,x,V}-B_{\theta,V}\cdot Z_V\}.
\end{align*}
Here $R_{\theta,x,V}$ means that the frozen reward state is set equal to the current
state.  With flow variables suppressed inside the integrands,
$U=\mathcal E_{\theta,\tau}(V)$ is equivalently determined by
\begin{align}
 U^1(\vartheta,t,y,\cdot)
 ={}&P^0_{t,T}F_\theta(\vartheta,y,\cdot)
 +\int_t^TP^0_{t,s}
 \{B_{\theta,V}\cdot D_xU^1+\Phi^1_{\theta,V}\}(s)\,ds,
                                                        \label{eq:supp-mild-E1}\\
 U^2(t,\cdot)
 ={}&P^0_{t,T}h_\theta
 +\int_t^TP^0_{t,s}
 \{B_{\theta,V}(s)\cdot D_xU^2(s,\cdot)\}\,ds.       \label{eq:supp-mild-E2}
\end{align}
Equations \eqref{eq:supp-gibbs-value}--\eqref{eq:supp-entropy-cancellation} prove
that fixed points of this map are exactly the original Gibbs-substituted evaluation
system.

\section{Strong coefficient and value norms}

For a metric space $(S,\rho)$ and a Banach space $E$, write
$[H]_{\gamma;\rho;E}:=\sup_{z\ne z'}
\|H(z)-H(z')\|_E/\rho(z,z')^\gamma$.  With
\begin{align*}
 \rho_r((y,t,x),(y',t',x'))
 &:=|y-y'|+|t-t'|^{1/2}+|x-x'|,\\
 \rho_{\rm ref}((\vartheta,y),(\vartheta',y'))
 &:=|\vartheta-\vartheta'|^{1/2}+|y-y'|,
\end{align*}
and $F^\sharp(\vartheta,y)=F(\vartheta,y,\cdot)$, the classical coefficient norm
used in the main paper is
\begin{align}
 \|M\|_{\mathfrak C^{\rm cl}_\gamma}
 :={}&\|M\|_{\mathfrak C^1}
 +\sup_a\|b(\cdot,\cdot,a)\|_{C_b^{\gamma/2,1+\gamma}(\mathcal Q)}\notag\\
 &+\sup_a\bigl([r(\cdot,\cdot,\cdot,a)]_{\gamma;\rho_r;\R}
   +[D_xr(\cdot,\cdot,\cdot,a)]_{\gamma;\rho_r;\R^d}\bigr)\notag\\
 &+\sup_{\vartheta,y}\|F^\sharp(\vartheta,y)\|_{C_b^{2+\gamma}(\R^d)}
 +[F^\sharp]_{\gamma;\rho_{\rm ref};C_b^{2+\gamma}(\R^d)}
 +\|h\|_{C_b^{2+\gamma}(\R^d)}.                       \label{eq:supp-classical-coefficient-space}
\end{align}

For $\mathbf z=(\vartheta,t,y,x)$ and $\mathbf z'$ in $\mathcal D_1$, put
\[
 \rho_*(\mathbf z,\mathbf z')=|t-t'|^{1/2}+|x-x'|
  +|\vartheta-\vartheta'|^{1/2}+|y-y'|,
\]
and let $[\cdot]_{\eta,*}$ be the corresponding bounded H\"older seminorm.  The
strong value norm is
\begin{align}
 \|V^1\|_{\mathfrak V^1_\eta}:={}&
 \sup_{\vartheta,y}
 \|V^1(\vartheta,\cdot,y,\cdot)\|_
 {C_b^{1+\eta/2,\,2+\eta}([\vartheta,T]\times\R^d)}
 +[V^1]_{\eta,*}+[D_xV^1]_{\eta,*},\notag\\
 \|V^2\|_{\mathfrak V^2_\eta}:={}&
 \|V^2\|_{C_b^{1+\eta/2,\,2+\eta}(\mathcal Q)}.       \label{eq:supp-strong-value-space}
\end{align}
The terminal-time norms use the one-sided parabolic metric.  Completeness follows
from completeness of the component H\"older spaces.  On diagonal points,
$\rho_*$ is at most twice the parabolic distance on $\mathcal Q$, and hence
\begin{equation}
 \|\mathsf T V^1\|_{C_b^{\eta/2,\eta}(\mathcal Q)}
 \le C\|V^1\|_{\mathfrak V^1_\eta}.                  \label{eq:supp-diagonal-trace}
\end{equation}

\section{Bounded-drift and stability closures}

For the bounded-drift lemma in the main paper, Duhamel's formula relative to
$P^0$ is
\[
 u(t)=P^0_{t,T}g+
 \int_t^TP^0_{t,s}\{c(s)\cdot D_xu(s)+f(s)\}\,ds.
\]
If $e(t)=\|u(t)\|_\infty+\|D_xu(t)\|_\infty$ and
$k(r)=1+r^{-1/2}$, the base semigroup bounds give
\[
 e(t)\le C\|g\|_{\mathcal B_x^1}
 +C\int_t^Tk(s-t)\|f(s)\|_\infty\,ds
 +CB_c\int_t^Tk(s-t)e(s)\,ds.
\]
On $[0,T]$, $k(r)\le C_Tr^{-1/2}$.  After absorbing the normalization of the
right-sided half integral into $C$, its iterates satisfy
\[
 \big\|(CB_cI_{T-}^{1/2})^m\big\|_{\infty\to\infty}
 \le \frac{(CB_c)^mT^{m/2}}{\Gamma(m/2+1)}.
\]
The series converges, proving the stated estimate and uniqueness.  Short-interval
contraction followed by backward continuation gives existence.

For the learned-model stability theorem, let $E(t)$ and $\varepsilon$ have the
meaning used in its proof and write $\Delta$ for a difference between the two
equilibria.  Directly from the action averages and
\eqref{eq:supp-gibbs-value},
\begin{align}
 \|\Delta B\|_\infty&\le C\{\varepsilon+d_{{\rm pol},t}\},\notag\\
 \|\Delta H_\tau\|_\infty&\le C\{\varepsilon+E(t)\},\notag\\
 \|\Delta\Phi^1\|_\infty
 &\le C\{\varepsilon+E(t)+d_{{\rm pol},t}\}.          \label{eq:supp-source-differences}
\end{align}
The constants use the common comparison envelope.  Keep one induced drift as the
principal coefficient in the difference equation.  The remaining drift forcing is
$\Delta B\cdot D_x\widetilde V$ and is controlled by the first line of
\eqref{eq:supp-source-differences}.  The bounded-drift estimate, followed by the
Gibbs policy bound in the main proof, gives its displayed half-order Volterra
inequality.  No spatial derivative of an induced drift is used.

\section{Fixed-temperature evaluation remainders}

At fixed temperature, direct differentiation gives
\begin{align}
 D\cG_\tau(q)[k](da)
 &=\frac{\pi_q(da)}\tau\{k(a)-\pi_qk\},                 \label{eq:supp-Gibbs-D1}\\
 D^2\cG_\tau(q)[k,\ell](da)
 &=\frac{\pi_q(da)}{\tau^2}
 \{(k(a)-\pi_qk)(\ell(a)-\pi_q\ell)
       -\operatorname{Cov}_{\pi_q}(k,\ell)\}.          \label{eq:supp-Gibbs-D2}
\end{align}
Thus $\|D\cG_\tau(q)[k]\|_{\rm TV}\le2\tau^{-1}\|k\|_\infty$ and
$\|D^2\cG_\tau(q)[k,\ell]\|_{\rm TV}
\le8\tau^{-2}\|k\|_\infty\|\ell\|_\infty$.  Also
$D^2H_\tau(q)[k,\ell]=\tau^{-1}\operatorname{Cov}_{\pi_q}(k,\ell)$.

\begin{proposition}[Mild evaluation remainders]
\label{prop:supp-parabolic-remainder}
Under the data and diffusion assumption of the main paper, fix $\tau>0$ and a $C^2$
coefficient family.  Write
$z=(\theta,V)$, equip $\R^p\times\mathfrak V$ with its product norm, and set
$U(z)=\mathcal E_{\theta,\tau}(V)$.  On a sufficiently small convex neighborhood of a
base point, there are $C_\tau<\infty$ and a modulus
$\omega_\tau(r)\downarrow0$ such that, with
$\mathcal Z=\R^p\times\mathfrak V$,
\begin{align}
 \|U(z+\zeta)-U(z)\|_{\mathfrak V}
 &\le C_\tau\|\zeta\|,                                      \label{eq:supp-E-lipschitz}\\
 \|U(z+\zeta)-U(z)-DU(z)[\zeta]\|_{\mathfrak V}
 &\le C_\tau\|\zeta\|^2,                                  \label{eq:supp-E-quadratic}\\
 \|DU(z+\zeta)-DU(z)\|_{\mathcal L(\mathcal Z,\mathfrak V)}
 &\le C_\tau\|\zeta\|,                                    \label{eq:supp-DE-lipschitz}\\
 \|DU(z+\zeta)-DU(z)-D^2U(z)[\zeta,\cdot]\|_
       {\mathcal L(\mathcal Z,\mathfrak V)}
 &\le\omega_\tau(\|\zeta\|)\|\zeta\|,                  \label{eq:supp-D2-remainder}\\
 \|D^2U(z+\zeta)-D^2U(z)\|_{\mathcal L^2(\mathcal Z;\mathfrak V)}
 &\le\omega_\tau(\|\zeta\|).                            \label{eq:supp-D2-continuity}
\end{align}
Here $\mathcal L^2(\mathcal Z;\mathfrak V)$ denotes bounded bilinear maps from
$\mathcal Z\times\mathcal Z$ to $\mathfrak V$.  No constant or modulus is asserted
to be uniform as $\tau\downarrow0$.
\end{proposition}

\begin{proof}
For either value component, write the mild equation as
\[
 U_z(t)=P^0_{t,T}g_z+\int_t^TP^0_{t,s}
       \{c_z\cdot D_xU_z+f_z\}(s)\,ds,
\]
where $c_z=B_{\theta,V}$ and $f_z$ is the relevant cancelled source.  Equations
\eqref{eq:supp-Gibbs-D1}--\eqref{eq:supp-Gibbs-D2} and the bounded diagonal trace show
that $(c_z,f_z,g_z)$ is $C^2$ in the required supremum norms.
Indeed, for $Z_V=\mathsf T V^1+G_z(V^2)D_xV^2$ and value directions
$W,\widehat W$,
\begin{align*}
 DZ_V[W]={}&\mathsf T W^1
 +G_{zz}(V^2)W^2D_xV^2+G_z(V^2)D_xW^2,\\
 D^2Z_V[W,\widehat W]={}&G_{zzz}(V^2)W^2\widehat W^2D_xV^2\\
 &+G_{zz}(V^2)
 \{W^2D_x\widehat W^2+\widehat W^2D_xW^2\}.
\end{align*}
The fixed-temperature Gibbs formulas, multiplication in the supremum spaces,
and integration against bounded action sections now give the asserted $C^2$
coefficient maps.  Uniform continuity of $G_{zzz}$ on bounded value ranges makes
their second derivatives continuous.

For directions $h_i$, put $c_i=Dc_z[h_i]$, $c_{12}=D^2c_z[h_1,h_2]$, and use the
same notation for $f$ and $g$.  The first and second variations are the unique mild
solutions
\begin{align}
 W_i(t)={}&P^0_{t,T}g_i+\int_t^TP^0_{t,s}
 \{c_z\cdot D_xW_i+c_i\cdot D_xU_z+f_i\}(s)\,ds,
                                                        \label{eq:supp-first-variation}\\
 Y_{12}(t)={}&P^0_{t,T}g_{12}+\int_t^TP^0_{t,s}
 \{c_z\cdot D_xY_{12}+c_{12}\cdot D_xU_z\notag\\
 &\hspace{8em}+c_1\cdot D_xW_2+c_2\cdot D_xW_1+f_{12}\}(s)\,ds.
                                                        \label{eq:supp-second-variation}
\end{align}
On the local input neighborhood,
$\|c_z\|_\infty\le\sup_a\|b(\cdot,a)\|_\infty$.
The bounded-drift estimate from the main paper therefore applies to all these
equations with one common principal-resolvent constant and no additional
temperature factor.
The bounded-drift estimate gives
$\|W_i\|_{\mathfrak V}\le C_\tau\|h_i\|$ and
$\|Y_{12}\|_{\mathfrak V}\le C_\tau\|h_1\|\|h_2\|$.

Let $\Delta U=U(z+\zeta)-U(z)$ and $W=DU(z)[\zeta]$.  Keeping $c_{z+\zeta}$ in the
principal resolvent, the remainder $R=\Delta U-W$ has terminal value
$g_{z+\zeta}-g_z-Dg_z\zeta$ and source
\begin{equation}\label{eq:supp-first-remainder-source}
 \{c_{z+\zeta}-c_z-Dc_z\zeta\}\cdot D_xU_z
 +(c_{z+\zeta}-c_z)\cdot D_xW
 +f_{z+\zeta}-f_z-Df_z\zeta.
\end{equation}
It is $O(\|\zeta\|^2)$, which proves
\eqref{eq:supp-E-lipschitz}--\eqref{eq:supp-E-quadratic}.  Subtracting the tangent
equations at $z+\zeta$ and $z$ gives \eqref{eq:supp-DE-lipschitz}.

For a unit direction $k$, put
$\mathcal R_\zeta[k]=DU(z+\zeta)[k]-DU(z)[k]-Y_{\zeta,k}$ and keep
$c_{z+\zeta}$ in the principal resolvent.  Its forcing is, up to the common sign in
the backward mild convention,
\begin{align}
 &\{Dc_{z+\zeta}-Dc_z-D^2c_z[\zeta,\cdot]\}[k]\cdot D_xU_z\notag\\
 &\quad+\{Dc_{z+\zeta}-Dc_z\}[k]\cdot D_x\Delta U
   +Dc_z[k]\cdot D_x(\Delta U-W)\notag\\
 &\quad+\{c_{z+\zeta}-c_z-Dc_z[\zeta]\}\cdot D_xDU(z)[k]
   +(c_{z+\zeta}-c_z)\cdot D_xY_{\zeta,k}\notag\\
 &\quad+\{Df_{z+\zeta}-Df_z-D^2f_z[\zeta,\cdot]\}[k].
                                                        \label{eq:supp-derivative-source}
\end{align}
The terminal value is
\[
 Dg_{z+\zeta}[k]-Dg_z[k]-D^2g_z[\zeta,k].
\]
After enlarging the modulus, its norm is at most
$\omega_\tau(\|\zeta\|)\|\zeta\|\|k\|$.  Continuity of the second coefficient
derivatives and the first two estimates give the same bound for the forcing in
\eqref{eq:supp-derivative-source}.  The bounded-drift estimate of the main paper then
proves
\eqref{eq:supp-D2-remainder}.  Subtracting the two second-variation equations at
$z+\zeta$ and $z$, continuity of $D^2(c,f,g)$ and the same bounded-drift estimate give
\eqref{eq:supp-D2-continuity}.  We may therefore define
$D^2U(z)[h_1,h_2]=Y_{12}$; the last estimate proves that this second Fr\'echet
derivative is continuous.
\end{proof}

\section{Fixed-temperature scope and a state-dependent base class}

The open class in the main paper contains explicit state-dependent controlled
perturbations.  Let $\A=[-1,1]$ with normalized Lebesgue measure,
$\sigma\equiv I_d$, take bounded $C_b^{3+\gamma}$ functions, and set
\begin{align*}
 b_\theta(t,x,a)&=\bar b(t,x)+\theta_2a\,\beta(t,x),\\
 r_\theta(y,t,x,a)&=\bar r(y,t,x)+\theta_1a\,\rho(y,t,x).
\end{align*}
If $\rho(x,t,x)$ is not spatially constant, then
\[
 D_{\theta_1}\pi_0(t,x,da)
 =\tau^{-1}a\rho(x,t,x)\mu(da)
\]
is state dependent, while the $\theta_2$ direction perturbs the controlled drift.
At the symmetric base point, its first-order averaged-drift derivative vanishes
because $\int_\A a\,\mu(da)=0$.  Constant diffusion and bounded smooth data directly
verify all clauses of the data assumption in the main paper, so no external EEHJB
existence theorem is used.

The implicit branch is a fixed-temperature statement.  After shrinking its local
neighborhood $U_\tau$, twice differentiability gives a constant $C_\tau<\infty$ with
\begin{equation}\label{eq:supp-quadratic-equilibrium-remainder}
 \|\Psi_\tau(\theta_0+v)-\Psi_\tau(\theta_0)
   -D\Psi_\tau(\theta_0)[v]\|
 \le C_\tau\|v\|^2.
\end{equation}
Consequently, a triangular array $\tau_n\downarrow0$ requires, in addition to
$\widehat\theta_n-\theta_0=O_{\mathbb P}(n^{-1/2})$,
\[
 \mathbb P(\widehat\theta_n\in U_{\tau_n})\longrightarrow1,
 \qquad C_{\tau_n}/\sqrt n\longrightarrow0.
\]
No bound on $C_\tau$ near zero follows from the fixed-temperature theorem.  Similarly,
a numerical equilibrium error $o_{\mathbb P}(n^{-1/2})$ preserves the limit law by
Slutsky's theorem, but a computable stopping rule requires a separate a posteriori
error estimate.

\section{Envelope verification for the bounded realization}

For the trigonometric model in Section~5 of the main paper, the uncontrolled
diffusion matrix is $\sigma^2I_d$, $|m_i|\le1$, and the uncontrolled evolution is
the heat semigroup.  The primitives are bounded trigonometric functions with bounded
derivatives of every order, and the induced drift is spatially constant.  With
$M(t,s)=\int_t^s m(u)\,du$, Gaussian convolution gives
\begin{align*}
 V^1(y;t,x)={}&\boldsymbol\varepsilon^\top e^{-\nu(T-t)}
     \sin\{x-y+M(t,T)\}\\
 &+\int_t^T e^{-\nu(s-t)}
     \sin\{x-y+M(t,s)\}^\top Km(s)\,ds\\
 &-\tau\int_t^T\KL(\pi_s\Vert\mu)\,ds,
 \qquad V^2=0.
\end{align*}
All flow-state and reference derivatives of this expression are bounded uniformly
for bounded $(K,\boldsymbol\varepsilon)$ and temperatures in the comparison set.
No such uniformity is claimed for parameter derivatives.  Moreover,
\[
 \tau\KL(\pi_t\Vert\mu)\le\osc_a(a^\top Z(t)),
 \qquad |Z(t)|\le C(|\boldsymbol\varepsilon|+T\|K\|),
\]
so the entropy contribution has the same uniform bound.  These estimates give the
comparison envelope in the main theorem.  At each fixed temperature the coefficient
maps are $C^\infty$, and the mild calculus follows from the heat-semigroup estimates.

For the policy-norm comparison in the graph corollary, each coordinate mean is a
bounded linear statistic, so
$|\int a_i\,d\dot\pi|\le\|\dot\pi\|_{\rm TV}$.  At the tied product law,
\[
 \frac{d\dot\pi[c]}{d\mu}(a)
 =\frac1v\sum_{i=1}^d a_i(\chi_\tau(t)c)_i.
\]
Therefore
\[
 \|\chi_\tau(t)\|_{\infty\to\infty}
 \le\left\|D_{\boldsymbol\varepsilon}\pi_t\big|_0
       \right\|_{\ell^\infty\to{\rm TV}}
 \le\frac d v\|\chi_\tau(t)\|_{\infty\to\infty}.
\]
This proves that the policy derivative and mean susceptibility have the same
polynomial order in the acyclic case and the same exponential rate in the cyclic case.

\section{Mesh-stiffness closure}

Use the notation of Proposition~6.1 in the main paper and put
$q_{N,r}=r!\binom Nr/N^r$.  If $N\to\infty$, then
$\max_{0\le r\le L}|q_{N,r}-1|\to0$, and coefficientwise comparison of the two
finite nilpotent sums gives $\mathfrak e_{\tau,N}\to0$ without coupling $N$ to
$\tau$.  Conversely, if $N\not\to\infty$, a subsequence has bounded integer $N$;
on a further subsequence it is constant, and the leading $K^L$ coefficient gives
$\mathfrak e_{\tau,N}\to1-q_{N,L}>0$.  When $L\le1$, the discrete and continuous
finite series are identical.

For the cyclic Perron mode, the logarithmic bounds in the main proof show that
$x^2/N\to0$ is sufficient.  Necessity follows because convergence of the response
ratio to one forces
\[
 \frac{x^2}{N(1+x/N)}\longrightarrow0.
\]
If $x/N$ failed to vanish, then along a subsequence $x/N\ge c>0$, and the last
quantity would be at least $cx/(1+c)$, which diverges as $x\to\infty$.
Thus $x/N\to0$, after which the displayed quantity is asymptotic to $x^2/N$.
Since $x=v\rho(K)T/\tau$, this is exactly $N\tau^2\to\infty$.

\section{Classical bootstrap}

\begin{proof}[Proof of the classical-regularity proposition in the main paper]
Let $V$ be a mild fixed point and fix $\eta\in(0,\gamma)$.  Constants below may depend
on the fixed temperature and on the local coefficient bounds.  The mild identity on
$[t,t']$, together with
$\|P^0_{t,t'}f-f\|_\infty\le C|t'-t|^{1/2}\|f\|_{\mathcal B_x^1}$,
gives, in the standard sup-norm time seminorm,
\begin{align}
 &\sup_{\vartheta,y}
 [V^1(\vartheta,\cdot,y,\cdot)]_{C_t^{1/2}C_x^0}
 +[V^2]_{C_t^{1/2}C_x^0}
 \le C_\tau(1+\|V\|_{\mathfrak V}).                  \label{eq:supp-value-time-holder}
\end{align}
Here $[u]_{C_t^{1/2}C_x^0}:=\sup_{t<t'}
\|u(t,\cdot)-u(t',\cdot)\|_\infty/|t'-t|^{1/2}$ on the relevant time interval.
For the base diffusion $X^{t,x}$, boundedness of its coefficient gives
\[
 \|P^0_{t,t'}f-f\|_\infty
 \le\|D_xf\|_\infty\sup_x\mathbb E|X_{t'}^{t,x}-x|
 \le C|t'-t|^{1/2}\|f\|_{\mathcal B_x^1}.
\]
The source integral on $[t,t']$ is $O(|t'-t|)$.

The singularity in the gradient estimates is $(s-t)^{-(1+\eta)/2}$, which is
integrable because $\eta<1$.  To record the time increment, write
$\mathcal F_V=B\cdot D_xV+\Phi$ for either bounded mild source.  For $t<t'$,
\begin{align*}
 D_xV(t)-D_xV(t')={}&
 \{D_xP^0_{t,T}-D_xP^0_{t',T}\}V(T)\\
 &+\int_t^{t'}D_xP^0_{t,s}\mathcal F_V(s)\,ds\\
 &+\int_{t'}^T\{D_xP^0_{t,s}-D_xP^0_{t',s}\}
                  \mathcal F_V(s)\,ds .
\end{align*}
The terminal term is controlled by the standard base linear Schauder estimate invoked
in Assumption~4.1(i) of the main paper; see
\cite[Sections~8.9--8.10]{Krylov1996}.  The first
integral is $O(|t-t'|^{1/2})$, and the assumed time-increment bound for the base
evolution controls the second by $C|t-t'|^{\eta/2}$.  Its spatial gradient bound
gives the corresponding spatial estimate.  Applying these bounds to both mild
equations gives
\begin{align}
 &\sup_{\vartheta,y}
 [D_xV^1(\vartheta,\cdot,y,\cdot)]_{C_b^{\eta/2,\eta}}
 +[D_xV^2]_{C_b^{\eta/2,\eta}}
 \le C_\tau(1+\|V\|_{\mathfrak V}).                    \label{eq:supp-flow-holder}
\end{align}
The seminorm in the first term is taken on
$[\vartheta,T]\times\R^d$ with the one-sided parabolic metric.
The terminal contributions are controlled up to $T$ by that base linear Schauder
estimate, using $F,h\in C_b^{2+\gamma}$; the singular gradient bound is needed only
for the source integral.

Take two frozen pairs $(\vartheta,y)$ and $(\vartheta',y')$.  Subtract their mild
equations on the common triangular domain
$t\ge\vartheta\vee\vartheta'$.  The coefficient differences are controlled by the
$\mathfrak C^{\rm cl}_\gamma$ norm.  Since $\eta<\gamma$, the terminal and bounded
source differences are $O(\rho_{\rm ref}^\eta)$.  The bounded-drift estimate yields
\begin{align}
 &\sup_{\substack{(\vartheta,y)\ne(\vartheta',y')\\
                  t\ge\vartheta\vee\vartheta'}}
 \frac{\|V^1(\vartheta,t,y,\cdot)
          -V^1(\vartheta',t,y',\cdot)\|_{\mathcal B_x^1}}
 {\rho_{\rm ref}((\vartheta,y),(\vartheta',y'))^\eta}
 \le C_\tau(1+\|V\|_{\mathfrak V}).                  \label{eq:supp-reference-holder}
\end{align}
For points with different flow times, first use
\eqref{eq:supp-value-time-holder} and \eqref{eq:supp-flow-holder} to move the earlier
point to the later flow time, then apply
\eqref{eq:supp-reference-holder}.  One-sided continuity covers the boundary of the
triangle.

A diagonal increment changes the flow and frozen variables together.  Combining
\eqref{eq:supp-flow-holder} and \eqref{eq:supp-reference-holder} gives
\begin{equation}
 [\mathsf T V^1]_{C_b^{\eta/2,\eta}(\mathcal Q)}
 \le C_\tau(1+\|V\|_{\mathfrak V}).                   \label{eq:supp-diagonal-holder}
\end{equation}
At fixed $\tau$, the first two Gibbs derivatives are bounded on bounded score sets.
The coefficient assumptions and
\eqref{eq:supp-value-time-holder} and \eqref{eq:supp-diagonal-holder} therefore put
the induced drift in $C_b^{\eta/2,\eta}$.  For each frozen pair, the cancelled source
has the same flow regularity.  More precisely, writing
$\Phi^1_{\vartheta,y}$ for that source,
\begin{equation}
 \|B\|_{C_b^{\eta/2,\eta}(\mathcal Q)}
 +\sup_{\vartheta,y}\|\Phi^1_{\vartheta,y}\|_
 {C_b^{\eta/2,\eta}([\vartheta,T]\times\R^d)}
 \le C_\tau(1+\|V\|_{\mathfrak V}).                 \label{eq:supp-source-holder}
\end{equation}

For each frozen pair, whole-space parabolic existence and the global Schauder estimate
give a classical solution $U$ with
\begin{equation}
 \|U\|_{C_b^{1+\eta/2,\,2+\eta}}
 \le C\left\{\|U(T)\|_{C_b^{2+\eta}}
 +\|\partial_tU+\mathcal L_t^0U+B\cdot D_xU\|_
 {C_b^{\eta/2,\eta}}\right\}.                         \label{eq:supp-schauder}
\end{equation}
After reversing time, the whole-space Cauchy problem has bounded uniformly elliptic
$C_b^{\eta/2,\eta}$ diffusion and drift coefficients, a source in the same class,
and terminal data in $C_b^{2+\eta}$.  The existence and estimate therefore follow
from \cite[Sections~8.9--8.10]{Krylov1996}.  The classical solution has the same mild
representation as $V$; uniqueness in the bounded-drift mild estimate gives $U=V$.
The Schauder constant is uniform in the frozen pair because the ellipticity,
coefficient, terminal, and source bounds are uniform.  Combine these uniform flow
norms with \eqref{eq:supp-reference-holder}, using the same two-step comparison as for
the diagonal.  This gives $[V^1]_{\eta,*}$ and $[D_xV^1]_{\eta,*}$ without a Schauder
estimate for frozen-pair differences.  Hence $V\in\mathfrak V_\eta$.  Taking
$\eta<\gamma$ avoids an endpoint trace claim.
\end{proof}

\bibliographystyle{siamplain}
\bibliography{references}